\documentclass[article,12pt]{amsart}

\usepackage{amsfonts}
\usepackage{amsmath}
\usepackage{graphicx}
\usepackage{color}
\usepackage{epsfig}
\usepackage{url}
\usepackage{mathabx}

\numberwithin{equation}{section}
\theoremstyle{plain}
\newtheorem{thm}{Theorem}[section]
\newtheorem{rem}{Remark}[section]
\newtheorem{prop}{Proposition}[section]

\newtheorem{cor}{Corollary}[section]
\newtheorem{lem}{Lemma}[section]

\newcommand{\dE}{\mathbb{E}}
\newcommand{\dR}{\mathbb{R}}

\newcommand{\dP}{\mathbb{P}}

\newcommand{\cD}{\mathcal{D}}

\newcommand{\wh}{\widehat}

\newcommand{\wt}{\widetilde}
\newcommand{\wc}{\widecheck}
\newcommand{\ind}{\mbox{1}\kern-.25em \mbox{I}}

\def\build#1_#2^#3{\mathrel{\mathop{\kern 0pt#1}\limits_{#2}^{#3}}}

\def\videbox{\mathbin{\vbox{\hrule\hbox{\vrule height1ex \kern.5em\vrule height1ex}\hrule}}}
\def\demend{\hfill $\videbox$\\}

\newcommand{\abs}[1]{\left\vert#1\right\vert}

\email{Marie.Duroydechaumaray@math.u-bordeaux1.fr}

\begin{document}
\title[Large deviations for the squared radial Ornstein-Uhlenbeck process]
{Large deviations for the squared radial Ornstein-Uhlenbeck process \vspace{1ex}}
\author{Marie du ROY de CHAUMARAY}
\address{Universit\'e Bordeaux 1, Institut de Math\'{e}matiques de Bordeaux,
UMR 5251, 351 Cours de la Lib\'{e}ration, 33405 Talence cedex, France.}

\begin{abstract}
We establish large deviation principles for the couple of the maximum likelihood estimators of dimensional and drift coefficients in the generalised squared radial Ornstein-Uhlenbeck process. We focus our attention to the most tractable situation where the dimensional parameter $a>2$ and the drift parameter $b<0$. In contrast to the previous literature, we state large deviation principles when both dimensional and drift coefficients are estimated simultaneously.
\end{abstract}

\maketitle


\section{Introduction}
The generalized squared radial Ornstein-Uhlenbeck process, also known as the Cox-Ingersoll-Ross process, is the strong solution of the stochastic differential equation 
\begin{equation}
\mathrm{d}X_t = (a+bX_t) \mathrm{d}t + 2 \sqrt{X_t}\, \mathrm{d}B_t 
\end{equation}
where the initial state $X_0=x \geq 0$, the dimensional parameter $a>0$, the drift coefficient $b \in \dR$  and $(B_t)$ is a standard Brownian motion. 
The behaviour of the process has been widely investigated and depends on the values of both coefficients $a$ and $b$. We shall restrict ourself to the most tractable situation where $a>2$ and $b<0$. In this case, the process is ergodic and never reaches zero.

We estimate parameters $a$ and $b$ at the same time using a trajectory of the process over the time interval $[0,T]$. The maximum likelihood estimators (MLE) of $a$ and $b$ are given by:
\begin{equation}\displaystyle \wh{a}_T=\frac{\int_{0}^{T}{X_t \, \mathrm{d}t} \int_{0}^{T}{\frac{1}{X_t} \, \mathrm{d}X_t}-T X_T}{\int_{0}^{T}{X_t \, \mathrm{d}t} \int_{0}^{T}{\frac{1}{X_t} \, \mathrm{d}t}-T^2} \:\: \: \text{ and } \: \:
\displaystyle \wh{b}_T= \frac{X_T \int_{0}^{T}{\frac{1}{X_t} \, \mathrm{d}t} - T \int_{0}^{T}{\frac{1}{X_t} \, \mathrm{d}X_t}}{\int_{0}^{T}{X_t \, \mathrm{d}t} \int_{0}^{T}{\frac{1}{X_t} \, \mathrm{d}t}-T^2}.
\end{equation}
Overbeck \cite{Ov} has shown that $\wh{a}_T$ and $\wh{b}_T$ both converge almost surely to $a$ and $b$. In addition, he has proven that $$\sqrt{T}\begin{pmatrix}\wh{a}_T - a \\
\wh{b}_T -b \end{pmatrix} \xrightarrow{\mathcal{L}} \mathcal{N}(0,4C^{-1}) \:\: \text{ where } \:\: C= \begin{pmatrix}
 \frac{-b}{a-2} & 1 \\
 1 & -\frac{a}{b}
 \end{pmatrix}.$$
Moderate deviation results for $\wh{a}_T$ and $\wh{b}_T$ are achieved in \cite{GJ2}. In addition,
Zani \cite{Z} established large deviation principles (LDP) for the MLE of $a$ assuming $b$ known and, conversely, for the MLE of $b$ assuming $a$ known. Our goal is to extend her results to the case where both parameters are estimated simultaneously. Our method is also different and we explain how we have simplified her approach at the beginning of Section 2.2 and Section 4, using a new strategy introduced by Bercu and Richou in \cite{BR} for the study of the Ornstein-Uhlenbeck process with shift.

The paper is organised as follows. Section 2 is devoted to an LDP for the couple $(\wh{a}_T,\wh{b}_T)$, which is obtained via LDPs for two other couples of estimators constructed on the MLE. Before we prove those results, which is respectively the aim of Sections 5 and 6, we investigate in Section 3 LDPs for some useful functionals of the process and compute in Section 4 the normalized cumulant generating function of a given quadruplet, which is a keystone for every LDP we establish in this paper. Technical proofs are postponed to Appendix A to E.

\section{Main results}
We start by rewriting the estimators $\wh{a}_T$ and $\wh{b}_T$ in such a way that they are much easier to handle. We need to suppose the starting point $x>0$ to apply the well-known It\^o's formula to $\log X_T$. We obtain that
\begin{equation}
\displaystyle \int_0^T{\frac{1}{X_t} \, \mathrm{d}X_t} = \log X_T - \log x + 2 \int_0^T{\frac{1}{X_t} \, \mathrm{d}t}
\end{equation} 
which leads to
\begin{equation}\label{est1}
\displaystyle \wh{a}_T=\frac{ S_T \left(2 \, \Sigma_T +  L_T \right)-\frac{X_T}{T}}{V_T} \, \, \: \,\, \text{ and } \, \, \, \, \:
\displaystyle \wh{b}_T=\frac{(\frac{X_T}{T}-2)\, \Sigma_T-L_T}{V_T}
\end{equation}
where the denominator $V_T= S_T \, \Sigma_T -1$ with $$\displaystyle S_T=\frac{1}{T}\int_0^{T}{X_t \, \mathrm{d}t}\,\, \text{ and } \, \, \Sigma_T= \frac{1}{T}\int_0^{T}{\frac{1}{X_t} \, \mathrm{d}t}, $$  and $$\displaystyle L_T= \frac{\log X_T - \log x}{T}.$$
For the remaining of the paper, we suppose the starting point $x$ equal to $1$. This assumption does not change the large deviation results because both estimators, with and without $\log x$, are exponentially equivalent so that they share the same LDP. 
As the rate function of the LDP for the MLE will turn out to be not directly computable,
we first consider two couples of simplified estimators constructed from $(\wh{a}_T,\wh{b}_T)$ using the fact that $L_T$ and $X_T/T$ both tend to zero almost surely for $T$ going to infinity. For those estimators LDPs are more straightforward and will finally be involved in the computation of the rate function of the MLE $(\wh{a}_T,\wh{b}_T)$. All the LDPs established in this paper are satisfied with speed $T$.

\subsection{Simplified estimators}
A first strategy to propose simplified estimators of $a$ and $b$ is to remove the logarithmic term $L_T$ in the expression of $\wh{a}_T$ and $\wh{b}_T$ given by (\ref{est1}). This way, we obtain a new couple $(\wt{a}_T,\wt{b}_T)$ defined by
\begin{equation}
\wt{a}_T=\frac{2 \, S_T \, \Sigma_T-\frac{X_T}{T}}{V_T} \, \, \: \,\, \text{ and } \, \, \, \, \: \wt{b}_T=\frac{(\frac{X_T}{T}-2)\, \Sigma_T}{V_T}.
\end{equation}
It is clear that $\wt{a}_T$ and $\wt{b}_T$ converge almost surely to $a$ and $b$. Moreover, we also have the same Central Limit Theorem (CLT)
$$\sqrt{T}\begin{pmatrix}\wt{a}_T - a \\
\wt{b}_T -b \end{pmatrix} \xrightarrow{\mathcal{L}} \mathcal{N}(0,4C^{-1}).$$  The proof of this result can be found in appendix A.
We also state an LDP for the couple $(\wt{a}_T,\wt{b}_T)$  assuming both parameters $a$ and $b$ unknown. 

\begin{thm}\label{LDPcouple}
The couple $(\wt{a}_T,\wt{b}_T)$ satisfies an LDP with good rate function
$$J_{a,b}(\alpha,\beta) = \left\lbrace \begin{array}{ll}
  \displaystyle \frac{(a-2)^2\beta}{8(2-\alpha)} \left(1+\frac{(2-\alpha)b}{\beta(a-2)}\right)^2+2\beta-b & \text{if } \alpha >2,\frac{b}{3} \leq \beta <0 \\
  & \text{or if } \alpha<2, \beta>0,\\
  \vspace{2ex}
  \displaystyle \frac{(a-2)^2\beta}{8(2-\alpha)} \left(1+\frac{(2-\alpha)b}{\beta(a-2)}\right)^2-\frac{\beta}{4}\left(1-\frac{b}{\beta}\right)^2 & \text{if } \alpha >2, \beta \leq \frac{b}{3},\\
  \vspace{2ex}
  \displaystyle  -b & \text{if } (\alpha,\beta)=(2,0) ,\\ 
  \displaystyle +\infty & \text{otherwise. } 
\end{array}   \right.$$
\end{thm}

\begin{proof}
The proofs of this theorem and the two following corollaries are postponed to Section 5.
\end{proof}

We give the shape of this rate function in Figure ~\ref{ft1} below in the particular case $(a,b)=(4,-1)$ and over $[3,5] \times [-4,-0.5]$. One can notice that the rate function reaches zero at point $(4,-1)$.

\begin{figure}[!h]
\includegraphics[width=10cm, height=7cm]{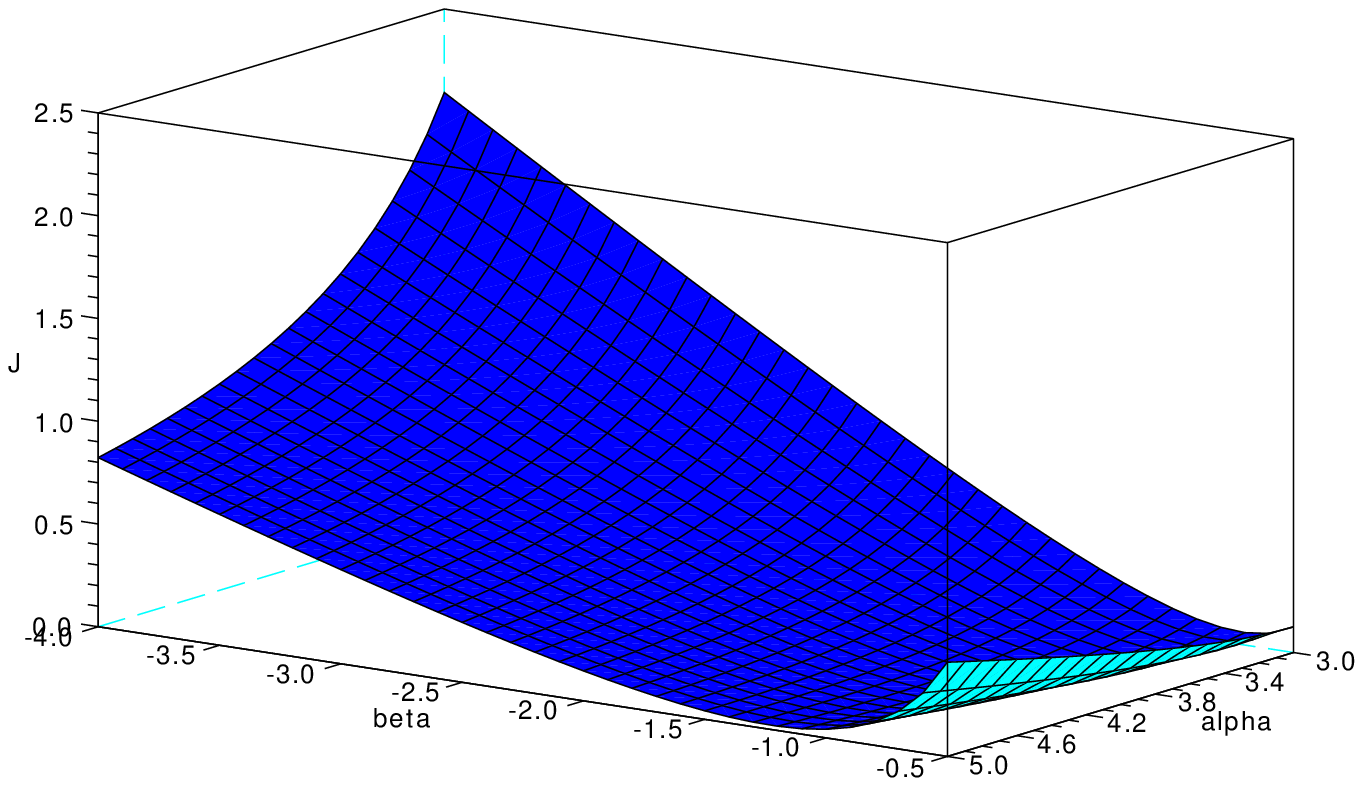}
\caption{Rate function for the couple of simplified estimators $(\wt{a}_T,\wt{b}_T)$}
\label{ft1}
\end{figure}

LDPs for each estimator $\wt{a}_T$ and $\wt{b}_T$ immediately follow from the contraction principle (see Theorem 4.2.1 of \cite{DeZ} and the following Remarks) which is recalled here for completeness.

\begin{lem}[Contraction Principle]\label{CP}
Let $\left(Z_T\right)_T$ be a sequence of random variables of $\dR^d$ satisfying an LDP with good rate function $I$ and $g:\dR^d \to \dR^n$ be a continuous function over $\mathcal{D}_{I}=\left\lbrace x \in \dR^d | I(x) <+\infty\right\rbrace$. The sequence $\left(g(Z_T)\right)_T$ satisfies an LDP with good rate function $J$ defined for all  $y \in \dR^n$ by
$$J\left(y\right)= \underset{\left\lbrace x \, \in \cD_I | g(x)=y \right\rbrace}{\inf} I\left(x\right).$$ 
\end{lem}

\begin{cor}\label{LDPa}
The sequence $(\wt{a}_T)$ satisfies an LDP with good rate function 
$$J_{a}(\alpha) = \left\lbrace \begin{array}{ll}
  \displaystyle  \frac{b}{4} \left(a-6-\sqrt{(a-2)^2+16(2-\alpha)} \right) & \text{if } \alpha \leq \ell_a,\\
  \displaystyle \frac{b}{4}\left(a-\sqrt{\alpha \left(\frac{(a-2)^2}{\alpha-2}+2\right)} \right) & \text{if } \alpha \geq \ell_a,
\end{array}   \right.$$
with $\ell_a= \frac{10}{9}+\frac{1}{9}\sqrt{64+9(a-2)^2}$.
\end{cor}

\begin{cor}\label{LDPb}
The sequence $(\wt{b}_T)$ satisfies an LDP with good rate function 
$$J_{b}(\beta) = \left\lbrace \begin{array}{ll}
  \displaystyle -\frac{\beta}{4}\left(1-\frac{b}{\beta}\right)^2 & \text{if } \beta \leq \frac{b}{3},\\
  \displaystyle 2\beta-b & \text{if } \beta \geq \frac{b}{3}.
\end{array}   \right.$$
\end{cor}

\begin{rem}
These rate function is the same than the one obtained by Zani \cite{Z} for the MLE of $b$ assuming $a$ known. 
\end{rem}

Figure ~\ref{ftab} displays in blue the rate functions $J_a$ and $J_b$ in the particular case where $(a,b)=(4,-1)$.

A second strategy to propose simplified estimators of $a$ and $b$ is to remove the term $X_T/T$ in the expression of $\wh{a}_T$ and $\wh{b}_T$ given by (\ref{est1}). Then, we obtain a new couple $(\wc{a}_T, \wc{b}_T)$ defined by 
\begin{equation}
\displaystyle \wc{a}_T=\frac{S_T \left(2 \, \Sigma_T +  L_T \right)}{V_T} \, \, \: \: \: \text{ and } \: \: \: \, \, \displaystyle \wc{b}_T=\frac{-2\, \Sigma_T-L_T}{V_T}.
\end{equation}
As previously, $\wc{a}_T$ and $\wc{b}_T$ converge almost surely to $a$ and $b$, and
$$\sqrt{T}\begin{pmatrix}\wc{a}_T - a \\
\wc{b}_T -b \end{pmatrix} \xrightarrow{\mathcal{L}} \mathcal{N}(0,4C^{-1}).$$ The proof of this result is given in appendix A. 
Again, we establish an LDP for the couple $(\wc{a}_T,\wc{b}_T)$, and deduce as corollaries LDPs for both estimators, assuming $a$ and $b$ unknown.

\begin{thm}\label{LDPcouple2}
The couple $(\wc{a}_T,\wc{b}_T)$ satisfies an LDP with good rate function
$$K_{a,b}(\alpha,\beta) = \left\lbrace \begin{array}{ll}
  \displaystyle  \frac{a}{4}\left(b-\beta\right)-\frac{\alpha}{8\beta} \left(b^2-\beta^2\right)- \frac{\beta}{\alpha}\left(\sqrt{2}+\sqrt{C_{\alpha}}\right)^2 & \text{if } \beta<0, 0< \alpha \leq \alpha_a \\
  & \text{or if } \beta>0, \alpha<0, \\
  \vspace{2ex}
  \displaystyle \frac{a}{4}\left(b-\beta\right)-\frac{\alpha}{8\beta} \left(b^2-\beta^2\right) - \frac{\beta \left(a-\alpha\right)^2}{8 \left(\alpha-2\right)}& \text{if } \beta<0, \alpha \geq \alpha_a,\\
  \vspace{2ex}
  \displaystyle - \frac{b}{4} \left(4-a+\sqrt{a^2+16}\right) & \text{if } (\alpha,\beta)=(0,0), \\ 
  \displaystyle +\infty & \text{otherwise, } 
\end{array}   \right.$$
where  $C_{\alpha}=\frac{1}{8} \left(a-\alpha\right)^2+2-\alpha$ and
$\alpha_a=-\frac{2}{3}\left(\frac{a}{2}-2-\sqrt{a^2-2a+4} \right)$.
\end{thm}

One can observe that the rate functions $J_{a,b}$ and $K_{a,b}$ are equal over some domain of $\mathbb{R}^2$. It is possible to see it on Figure ~\ref{ft2} below which is quite similar to the previous one and displays the rate function $K_{a,b}$ in the particular case where $(a,b)=(4,-1)$.

\begin{figure}[!h]
\includegraphics[width=10cm, height=7cm]{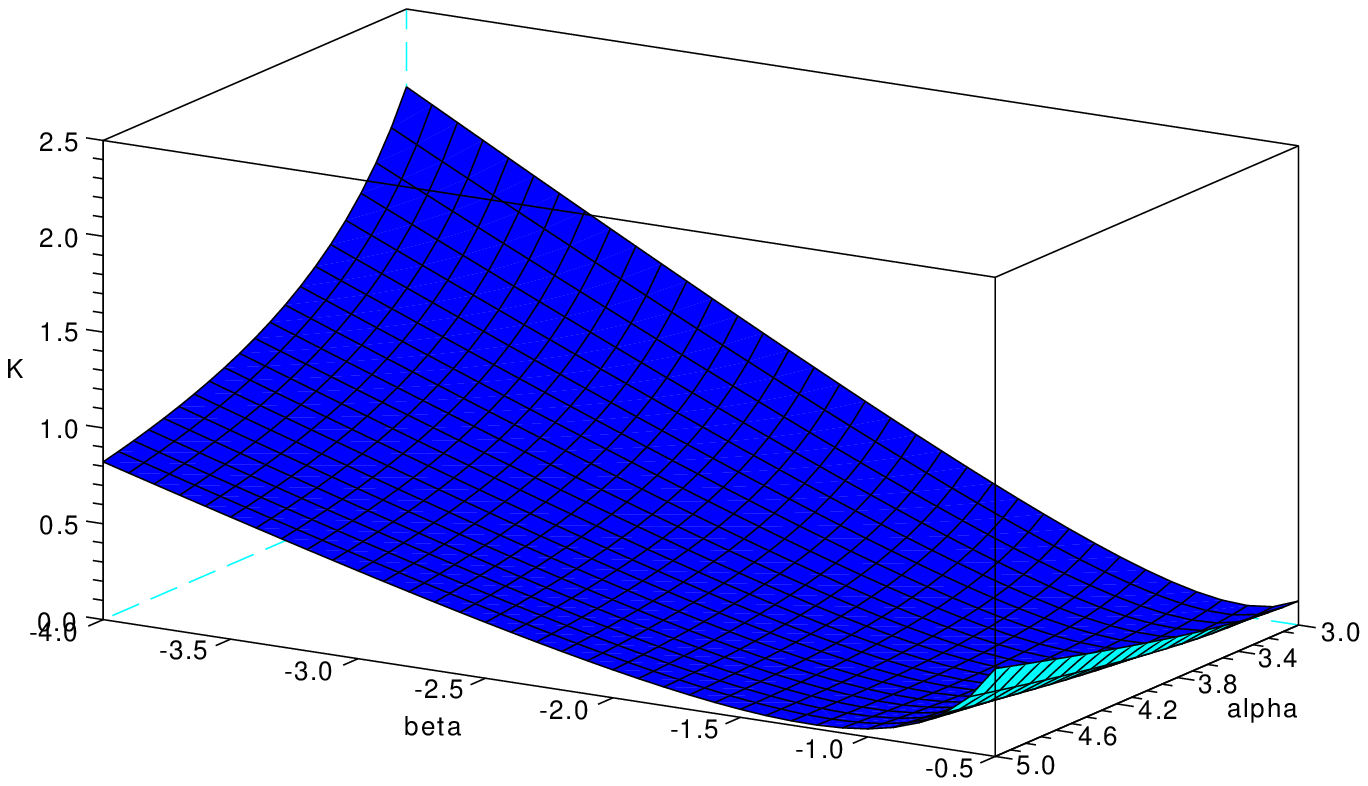}
\caption{Rate function for the couple of simplified estimators $(\wc{a}_T,\wc{b}_T)$.}
\label{ft2}
\end{figure}

\begin{cor}\label{LDPa2}
The sequence $(\wc{a}_T)$ satisfies an LDP with good rate function 
$$K_{a}(\alpha) = \left\lbrace \begin{array}{ll}
  \vspace{2ex} \displaystyle - \frac{b}{4} \left(4-a+\sqrt{a^2+16}\right) & \text{if } \alpha=0,\\
  \displaystyle K_{a,b}(\alpha,\beta_b) & \text{if } \alpha <\alpha_a, \alpha \neq 0,\\
  \displaystyle \frac{b}{4}\left(a-\sqrt{\alpha\left(\frac{(a-2)^2}{\alpha-2}+2\right)} \right) & \text{if } \alpha \geq \alpha_a
\end{array}   \right.$$
with $\beta_b= b\alpha \left(16\sqrt{2C_{\alpha}}+a^2-8\alpha+32\right)^{-1/2}$ and $\alpha_a$, $C_{\alpha}$ are defined in Theorem ~\ref{LDPcouple2}.
\end{cor}

\begin{cor}\label{LDPb2}
The sequence $(\wc{b}_T)$ satisfies an LDP with good rate function 
$$K_{b}(\beta) = \inf \Bigl\{ K_{a,b}(\alpha,\beta) \ \slash \ \alpha \in \dR \Bigr\}.$$
In particular, $K_{b}(0)=K_{a,b}\left(0,0\right) =- \frac{b}{4} \left(4-a+\sqrt{a^2+16}\right)$.
\end{cor}

Figure ~\ref{ftab} displays in green the rate functions $K_a$ and $K_b$ in the particular case where $(a,b)=(4,-1)$.

\subsection{Large deviation results for the MLE}

The next theorem gives a large deviation principle for the MLE $(\wh{a}_T,\wh{b}_T)$ of the couple $(a,b)$. In contrast with the previous literature, we consider both parameters unknown and estimate them simultaneously. We also simplified the approach of the previous literature as our proofs only rely on the G\"artner-Ellis theorem and do not need, for example, accurate time-depending changes of probability.

\begin{thm}\label{LDP_MLE}
The couple $(\wh{a}_T,\wh{b}_T)$ satisfies an LDP with good rate function $I_{a,b}$ given over $\dR^2$ by
 $$I_{a,b}(\alpha,\beta)=\min \left(J_{a,b}(\alpha,\beta),K_{a,b}(\alpha,\beta) \right).$$
 \end{thm}

\begin{proof}
Section 6 is devoted to the proof of this result.
\end{proof}

Making use of the contraction principle (see Lemma~\ref{CP}), once again, we obtain straightforwardly the two following corollaries.
\begin{cor}\label{LDPA}
The sequence $(\wh{a}_T)$ satisfies an LDP with good rate function 
$$I_{a}(\alpha)=\min \left(J_{a}(\alpha),K_{a}(\alpha) \right).$$
\end{cor}

\begin{cor}\label{LDPB}
The sequence $(\wh{b}_T)$ satisfies an LDP with good rate function 
$$I_{b}(\beta)=\min \left(J_{b}(\beta),K_{b}(\beta) \right).$$
\end{cor}

Figure ~\ref{ftab} displays in red the rate functions $I_a$ and $I_b$ in the particular case where $(a,b)=(4,-1)$.

\begin{figure}[!h]
\begin{minipage}[b] {1\linewidth}
\centering 
\centerline {\includegraphics[width=7cm,height=7cm,angle=0]{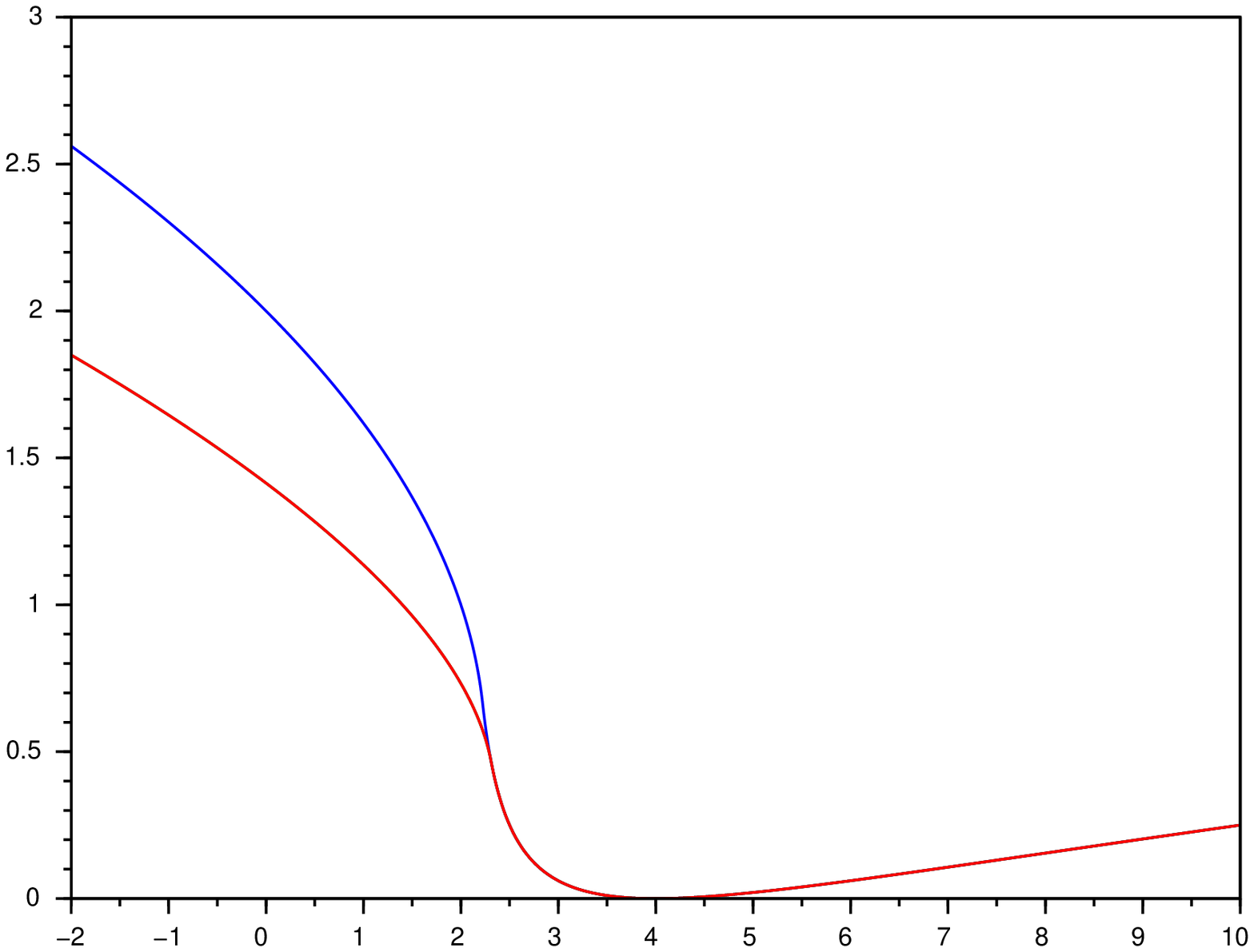} 
\includegraphics[width=7cm,height=7cm,angle=0]{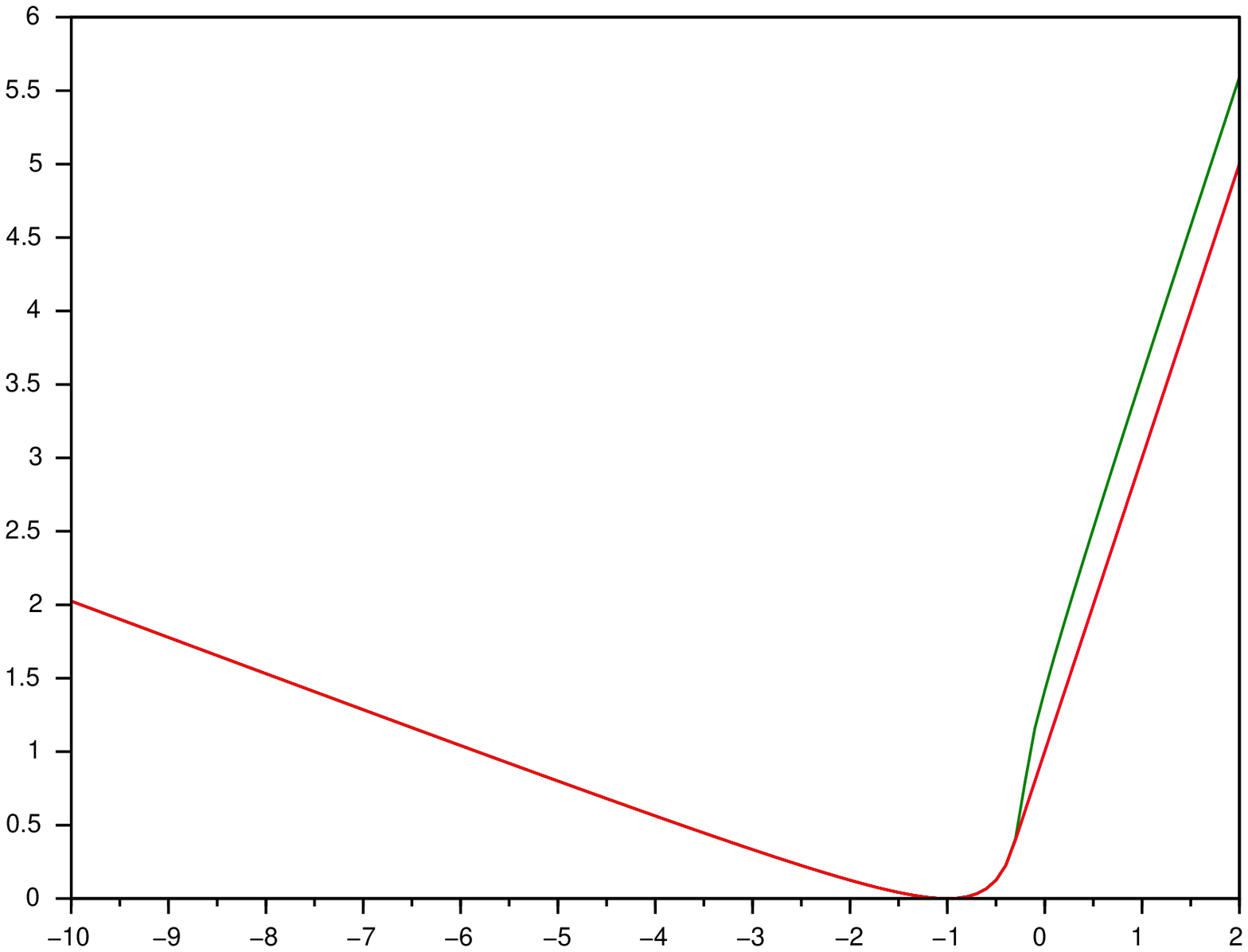} }
\end{minipage}
\caption{Rate functions for dimensional and drift parameters.}
\label{ftab}
\end{figure}

\begin{rem}
Both couples of simplified estimators perform better than the MLE in terms of large deviations, as their rate functions are always greater.
\end{rem}

\section{Some results about the process}

 The aim of this section is to establish LDPs with speed $T$ for $S_T$, $\Sigma_T$ and $V_T$, which will be involved in the proof of the main theorem.
 
 \begin{lem}\label{LDPC}
 The couple $\displaystyle \left(S_T,\Sigma_T\right)$ satisfies an LDP with good rate function $$I(x,y) =\left\lbrace \begin{array}{ll}
  \displaystyle \frac{y}{2(xy-1)}+\frac{b^2}{8}x+\frac{(a-2)^2}{8}y + \frac{ab}{4}  & \text{if } x>0, y>0 \text{ and } xy -1 > 0\\
  \displaystyle +\infty & \text{otherwise. } 
\end{array}   \right.$$ 
 \end{lem}

\begin{proof}
See appendix B.
\end{proof}
 
The following result can be proven either directly with the same method or using the previous lemma together with the contraction principle recalled in Lemma~\ref{CP}. 

 \begin{thm}\label{LDPmoyennes}
  The sequence $\left(S_T\right)$ satisfies an LDP with good rate function $$I(x) =\left\lbrace \begin{array}{lr}
   \frac{(a+bx)^2}{8x} & \text{if } x>0 \\
  +\infty & \text{if } x \leq 0.
\end{array}   \right.$$
 
 In addition, the sequence $\left(\Sigma_T\right)$ satisfies an LDP with good rate function $$J(x) =\left\lbrace \begin{array}{lr}
   \frac{\big((a-2)x+b\big)^2}{8x} & \text{if } x>0 \\
  +\infty & \text{if } x \leq 0.
\end{array}   \right.$$
 \end{thm}
 
It is now easy to establish an LDP for $V_T$. We recall that $V_T=h\left(S_T,\Sigma_T \right)$ where $h$ is the function defined on $\dR^2$ by $h(x,y)=xy-1$.

\begin{thm} \label{LDPdenominateur}
The sequence $\displaystyle \left(V_T\right)$ verifies an LDP with good rate function
$$K(x) = \left\lbrace \begin{array}{ll}
  \displaystyle -\frac{b}{4}\sqrt{(x+1) \left((a-2)^2+\frac{4}{x}\right)}+\frac{ab}{4} \, \, \, \, & \text{if } x>0\\
  \displaystyle +\infty & \text{if } x \leq 0.
\end{array}   \right.$$
\end{thm} 

 \begin{proof}
 It follows immediately from Lemma ~\ref{LDPC} together with the contraction principle (see Lemma~\ref{CP}). It only remains to explicitly evaluate the rate function $K$ given, for all real $z$, by 
 $$K(z)=  \underset{\{(x,y) | z=xy-1\}}{\inf} I(x,y)$$ where $I$ is defined in Lemma ~\ref{LDPC}.
 \end{proof}

\section{Cumulant generating function for the quadruplet}

 To establish LDPs for the estimators $\left(\wh{a}_T,\wh{b}_T\right)$, we need to compute the normalized cumulant generating function of the quadruplet $\displaystyle \left(X_T/T, S_T, \Sigma_T, L_T\right)$. However, this does not lead to a steep function (see \cite{DeZ} for the definition), which is a necessary condition to apply G\"artner-Ellis theorem . In contrast with the previous literature, we will not search another method to obtain large deviation results. Following the strategy of \cite{BR},  the idea to overcome this difficulty is to consider instead the quadruplet $\mathcal{Q}_T=\left(\sqrt{X_T/T}, S_T, \Sigma_T, \mathcal{L}_T\right)$, where \begin{equation}
 \mathcal{L}_T=-\sqrt{\frac{- \log X_T}{T}} \mathbf{1}_{X_T<1}+ \frac{\log X_T}{T} \mathbf{1}_{X_T \geq 1}.
 \end{equation}

\begin{prop}\label{CGFquadruplet}
Let $\Lambda_T(\lambda,\mu,\nu,\gamma)$ be the normalized cumulant generating function of the quadruplet $\mathcal{Q}_T$ given over $\dR^4$ by
$$\Lambda_T(\lambda,\mu,\nu, \gamma)= \frac{1}{T} \log \left(\mathbb{E}\left[\exp\left(\lambda \sqrt{T} \sqrt{X_T}+\gamma \, T \mathcal{L}_T+\mu \int_0^T{X_t \, \mathrm{d}t}+ \nu \int_0^T{\frac{1}{X_t} \, \mathrm{d}t}\right)\right]\right).$$
 Denote by $\Lambda$ its pointwise limit as $T$ tends to $+\infty$. For all $\lambda, \gamma \in \mathbb{R}$, $\mu < \frac{b^2}{8}$ and $\nu < \frac{(a-2)^2}{8}$,
$$\Lambda(\lambda,\mu,\nu, \gamma) = \left\lbrace \begin{array}{ll}
  \displaystyle -\frac{d}{2}(1+f)-\frac{ab}{4} +\frac{\lambda^2}{d-b} \, \, \, \, & \text{if } \lambda >0 \text{ and } \gamma \geq 0\\
  & \text{ or if } \gamma<0, \lambda>0 \text{ and } \frac{\gamma^2}{\lambda^2}< \frac{2f+a+2}{d-b}, \\
  \displaystyle -\frac{d}{2}(1+f)-\frac{ab}{4}+\frac{\gamma^2}{2f+a+2}  & \text{if } \lambda \leq 0 \text{ and } \gamma < 0\\
  & \text{ or if } \gamma<0, \lambda>0 \text{ and } \frac{\gamma^2}{\lambda^2} \geq \frac{2f+a+2}{d-b}, \\
  \displaystyle -\frac{d}{2}(1+f)-\frac{ab}{4}  & \text{if } \lambda \leq 0 \text{ and } \gamma \geq 0,\\
\end{array}   \right.$$
where $\displaystyle d=\sqrt{b^2-8\mu}$ and $\displaystyle f= \frac{1}{2} \sqrt{\left(a-2\right)^2-8\nu}$.
\end{prop}

\begin{lem}\label{steep}
The function $\Lambda$ is steep. 
\end{lem}

\begin{proof}
$\Lambda$ is differentiable over its domain $\mathcal{D}_{\Lambda}= \dR^4 \times [-\infty, \frac{b^2}{8}[ \times [-\infty, \frac{(a-2)^2}{8}[ \times \dR^4$ and its gradient is given by
\begin{equation}\label{gradient}\nabla \Lambda = \begin{pmatrix}
 \frac{2\lambda}{d-b} \textbf{1}_{\Delta_1}\\
 \frac{2(1+f)}{d}+  \frac{4\lambda^2}{d(d-b)^2}\textbf{1}_{\Delta_1}\\
 \frac{d}{2f} + \frac{2\gamma^2}{f(2f+a+2)^2} \textbf{1}_{\Delta_2}\\
 \frac{2\gamma}{2f+a+2} \textbf{1}_{\Delta_2}
\end{pmatrix},
\end{equation}
where $\Delta_1= \Bigl\{ (\lambda,\mu,\nu,\gamma) \in \mathcal{D}_{\Lambda} \ \slash \ \lambda >0 \text{ and } \gamma \geq 0 \text{ or }  \gamma<0, \lambda>0 \text{ and } \frac{\gamma^2}{\lambda^2}< \frac{2f+a+2}{d-b} \Bigr\}$ and $\Delta_2=\Bigl\{ (\lambda,\mu,\nu,\gamma) \in \mathcal{D}_{\Lambda} \ \slash \ \lambda \leq 0 \text{ and } \gamma < 0 \text{ or } \gamma<0, \lambda>0 \text{ and } \frac{\gamma^2}{\lambda^2} \geq \frac{2f+a+2}{d-b}\Bigr\} $.
We easily obtain that the norm of (\ref{gradient}) goes to infinity for any sequence in the interior of $\mathcal{D}_{\Lambda}$ converging to a boundary point. 
\end{proof}

\begin{proof}[Proof of Proposition~\ref{CGFquadruplet}]
We want to find the limit of $\Lambda_T(\lambda,\mu,\nu, \gamma)$ as $T \to +\infty$. It follows from Theorem 5.10 in \cite{CL} (with a misprint pointed out in \cite{KAB2}) that \begin{equation}\Lambda_T(\lambda,\mu,\nu,\gamma)= \frac{1}{T} \log \left(\int_0^\infty{e^{\lambda \sqrt{Ty}+\gamma \, T l(T,y)}p(T,x,y) \, \mathrm{d}y}\right)
\end{equation}
 where $l(T,y)=-\sqrt{\frac{- \log y}{T}} \,  \mathbf{1}_{y<1}+ \frac{\log y}{T} \, \mathbf{1}_{y \geq 1} $ and \begin{equation}
 \begin{array}{ll}
p(T,x,y)= & \displaystyle \frac{d\left(x/y\right)^{-\frac{a-2}{4}}}{4 \sinh\left(dT/2\right)}  \, I_{f}\left(\frac{d\sqrt{xy}}{2\sinh\left(dT/2\right)} \right)\\
& \displaystyle \times \exp \left(-\frac{1}{4} \left(abT+d(x+y)\coth\left(dT/2\right)+b(x-y)\right)\right)
\end{array}
\end{equation}
with $\displaystyle d=\sqrt{b^2-8\mu}$ and $\displaystyle f= \frac{1}{2} \sqrt{\left(a-2\right)^2-8\nu}$, $I_f$ being the modified Bessel function of the first kind.
We take out of the integral all the terms that do not depend on $y$. This leads to
$$\Lambda_T(\lambda,\mu, \nu, \gamma)= \frac{1}{T}\left(\log \mathcal{J}_T+ \log \left( \frac{d\sqrt{x}}{4 \sinh(dT/2)}\right)- \frac{1}{4}\left(abT+dx \coth\left(dT/2\right)+a \log(x) +bx\right)\right)$$
where \begin{equation}\displaystyle \mathcal{J}_T=\int_0^\infty{e^{\lambda \sqrt{Ty}+\gamma \, T l(T,y)-\frac{y}{4}(d\coth(dT/2)-b)} y^{\frac{a-2}{4}} \,  I_{f}\left(\frac{d\sqrt{xy}}{2\sinh(dT/2)} \right)\, \mathrm{d}y}.
\end{equation}
However, as soon as $T$ tends to infinity, $\coth(dT/2)$ goes to $1$, which implies that \begin{equation}
\lim_{T \to +\infty} - \frac{1}{4T}\left(abT+dx \coth\left(\frac{dT}{2}\right)+a \log(x) +bx\right) = -\frac{ab}{4}.
\end{equation}
On the other hand, 
$$
\frac{1}{T} \log \left(\sinh\left(\frac{dT}{2}\right)\right) = \frac{1}{T} \frac{dT}{2} +\frac{1}{T} \log \left(\frac{(1-e^{-dT})}{2} \right)
$$
which clearly leads to 
\begin{equation}\lim_{T \to +\infty} \frac{1}{T} \log \left( \frac{d\sqrt{x}}{4 \sinh(dT/2)}\right) = -\frac{d}{2}.
\end{equation}
We have to establish the asymptotic behaviour of $\frac{1}{T} \log \mathcal{J}_T$. We split $\mathcal{J}_T$ into two terms: $\mathcal{J}_T=H_T+K_T$ where \begin{equation}\label{defHT}H_T=\int_0^1{e^{\lambda \sqrt{Ty}-\gamma \sqrt{-T \log y}-\alpha_T y} y^{\frac{a-2}{4}} \,  I_{f}\left(\beta_T \sqrt{y} \right)\,  \mathrm{d}y},
\end{equation}
\begin{equation}\label{defKT}
K_T=\int_1^\infty {e^{\lambda \sqrt{Ty}-\alpha_T y} y^{\gamma + \frac{a-2}{4}} \,  I_{f}\left(\beta_T \sqrt{y} \right)\,  \mathrm{d}y}
\end{equation}
with \begin{equation}\alpha_T=\frac{d\coth(dT/2)-b}{4} \:\: \text{ and } \:\: \beta_T = \frac{d\sqrt{x}}{2\sinh(dT/2)}.
\end{equation}
We need the four following lemmas, whose proofs are postponed to Appendix C.

\begin{lem}\label{J1gammaNeg}
For all $\gamma <0$ and $\lambda \in \mathbb{R}$, one can find the following bounds for $H_T$ as $T$ goes to infinity. $$\displaystyle H_T \leq \frac{2^{1-f}}{\Gamma(f+1)}\sqrt{\pi} g^{-3/2} |\gamma| \sqrt{T} \, e^{|\lambda|+\alpha_T/T +\beta_T/\sqrt{T}}\beta_T^f \, \exp\left(\frac{\gamma^2 T}{4g}\right)$$ and
$$\displaystyle H_T \geq \frac{2^{-1-f}}{\Gamma(f+1)}\sqrt{\pi} g^{-3/2} |\gamma| \sqrt{T} \, e^{-|\lambda|-\alpha_T/T} \beta_T^f \, \exp\left(\frac{\gamma^2 T}{4g}\right),$$ 
 where $g=\frac{2f+a+2}{4}$.
\end{lem}

\begin{lem}\label{J1gammaPos}
For all $\gamma \geq 0$ and $\lambda \in \mathbb{R}$, bounds for $H_T$ are given by
$$\displaystyle H_T \leq  \frac{\left(\beta_T\right)^f}{\Gamma(f+1) 2^f} \, \displaystyle \exp\left({|\lambda|\sqrt{T}+\beta_T}\right)$$ 
and
$$H_T \geq \frac{\left(\beta_T\right)^f\, \, \varepsilon_T}{2^f \, \Gamma\left(1+f\right)} \, \exp \left(-\gamma \sqrt{T} \sqrt{\log T}- g \, \log T \right),$$
where $\varepsilon_T=e^{-\alpha_T} \left(\varepsilon+\frac{e^{-|\lambda|}}{g} \left(1-\frac{\gamma \sqrt{T}}{2g\sqrt{\log T} +\gamma \sqrt{T}}\right)\right)$ and $\varepsilon=\exp\left(\lambda \mathbf{1}_{\lambda \geq 0}+\lambda \sqrt{T} \mathbf{1}_{\lambda < 0}\right)$.
\end{lem}

\begin{lem}\label{J2lambdaNeg}
For all $\lambda \leq 0$ and $\gamma \in \mathbb{R}$ and $T$ tending to infinity, $\displaystyle K_T=O\left(\left(\beta_T\right)^ f\right)$. Moreover if $\gamma \geq 0$, we have the following lower bound
$$K_T \geq \frac{2^{1-f} \left(\beta_T\right)^f}{\Gamma(1+f)} \frac{e^{-\alpha_T}}{2\alpha_T} \frac{1}{1-\frac{\lambda \sqrt{T}}{2\alpha_T}} \exp \left(\lambda \sqrt{T}\right).$$
\end{lem}

\begin{lem}\label{J2lambdaPos}
For all $\lambda >0$ and $\gamma \in \mathbb{R}$, $K_T$ is bounded as follows for $T$ going to infinity,
$$K_T \leq \frac{2^{2-f} \sqrt{2\pi} \left(\beta_T\right)^f}{\Gamma(1+f)} \left(\beta_T+\lambda\sqrt{T}\right)^{2\gamma+2g-1} \, \exp\left({\frac{(\lambda\sqrt{T}+\beta_T)^2}{4\alpha_T}}\right)$$ and $$K_T \geq 2^{1-f} \sqrt{\frac{\pi}{d-b}} \frac{\left(\beta_T\right)^f}{\Gamma(1+f)} \, m_{\gamma,\lambda,T} \,  \exp\left({\frac{\lambda^2 T}{4\alpha_T}}\right).$$
where $m_{\gamma,\lambda,T}= \min \left\lbrace\left(\frac{\lambda \sqrt{T}}{\alpha_T}\right)^{2(\gamma+g-1)}; \frac{\lambda \sqrt{T}}{\alpha_T}\right\rbrace$.
\end{lem}

It is clear with those lemmas that the asymptotic behaviour of $\mathcal{J}_T$ depends on the sign of $\lambda$ and $\gamma$: \\
$\bullet$ For all $\lambda>0$ and $\gamma \geq 0$: we directly deduce from Lemmas ~\ref{J2lambdaPos} and ~\ref{J1gammaPos} that, for $T$ large enough,
$$H_T+K_T \leq \frac{2^{2-f} \sqrt{2\pi} \left(\beta_T\right)^f}{\Gamma(1+f)} \left(\beta_T+\lambda\sqrt{T} \right)^{2\gamma+2g-1}\left(e^{\frac{(\lambda\sqrt{T}+\beta_T)^2}{4\alpha_T}}+e^{|\lambda| \sqrt{T} +\beta_T} \right)$$
and thus \begin{equation}\displaystyle \varlimsup_{T \to \infty} \frac{1}{T}\log \mathcal{J}_T \leq f \lim_{T \to \infty}{\frac{1}{T}\log \beta_T} + \lim_{T \to \infty}{\frac{1}{T}\frac{(\lambda\sqrt{T}+\beta_T)^2}{4\alpha_T}}=-f\, \frac{d}{2} +  \frac{\lambda^2}{d-b}.
\end{equation} 
We show alike by using the lower bounds of Lemmas ~\ref{J2lambdaPos} and ~\ref{J1gammaPos} that \begin{equation}\varliminf_{T \to \infty}\frac{1}{T}\log \mathcal{J}_T \geq -f\, \frac{d}{2} +  \frac{\lambda^2}{d-b}
\end{equation} and we finally obtain
\begin{equation} \lim_{T \to \infty}\frac{1}{T}\log \mathcal{J}_T =-f\, \frac{d}{2} +  \frac{\lambda^2}{d-b}.
\end{equation}\\
$\bullet$ For all $\lambda \leq 0$ and $\gamma <0$:  With Lemma ~\ref{J2lambdaNeg}, we know that $\displaystyle K_T=O\left(\left(\beta_T\right)^ f\right)$. Thus, $H_T+K_T=H_T + O\left(\left(\beta_T\right)^ f\right)$. Lemma ~\ref{J1gammaNeg} gives bounds for $H_T$, which lead to $$\varlimsup_{T \to \infty} \frac{1}{T}\log \mathcal{J_T} \leq \lim_{T \to \infty} \frac{1}{T}\log \left(\left(\beta_T\right)^f h_{T,\lambda,\gamma} \exp\left(\frac{\gamma^2 T}{4g}\right)\right) +\lim_{T \to \infty} \frac{1}{T}\log \left(1+ C h_{T,\lambda,\gamma}^{-1} e^{-\frac{\gamma^2 T}{4g}}\right)$$ where $C$ is some positive constant and $$h_{T,\lambda,\gamma}=\frac{2^{1-f}}{\Gamma(f+1)}\sqrt{\pi} g^{-3/2} |\gamma| \sqrt{T} \, e^{|\lambda|+\alpha_T/T +\beta_T/\sqrt{T}}.$$ Using the fact that $h_{T,\lambda,\gamma}^{-1} e^{-\frac{\gamma^2 T}{4g}}$ tends to zero as $T$ goes to infinity, we obtain
\begin{equation} \varlimsup_{T \to \infty} \frac{1}{T}\log \mathcal{J_T} \leq  -f\, \frac{d}{2}+\frac{\gamma^2}{4g}.
\end{equation}
We obtain the same lower bound by using the lower bound in Lemma ~\ref{J1gammaNeg}.\\
$\bullet$ For all $\lambda \leq 0$ and $\gamma \geq 0$: Lemma ~\ref{J1gammaPos} gives  $H_T= O\left(\left(\beta_T\right)^f \exp\left({|\lambda|\sqrt{T}}\right)\right)$ and by Lemma ~\ref{J2lambdaNeg}, we know that $\displaystyle K_T=O\left(\left(\beta_T\right)^ f\right)$. Consequently,
\begin{equation}\varlimsup_{T \to \infty} \frac{1}{T}\log \mathcal{J}_T \leq f \lim_{T \to \infty} \frac{1}{T}\log \beta_T = -f\, \frac{d}{2}.
\end{equation}
And the lower bounds given in Lemmas ~\ref{J1gammaPos} and ~\ref{J2lambdaNeg} lead to  \begin{equation}
\varliminf_{T \to \infty} \frac{1}{T}\log \mathcal{J}_T \geq f \lim_{T \to \infty} \frac{1}{T}\log \beta_T =-f\, \frac{d}{2}.
\end{equation}\\
$\bullet$ For all $\lambda >0$ and $\gamma <0$: using Lemmas ~\ref{J2lambdaPos} and ~\ref{J1gammaNeg}, we show that
$$H_T+K_T \leq C \left(\beta_T\right)^f \left(\overline{h}_{T,\lambda}\, \exp\left(\frac{\gamma^2 T}{4g}\right) +   \left(\beta_T+\lambda\sqrt{T}\right)^{2\gamma+2g-1} \, \exp\left({\frac{(\lambda\sqrt{T}+\beta_T)^2}{4\alpha_T}}\right)  \right) $$
where $\overline{h}_{T,\lambda}=\sqrt{T} \, e^{|\lambda|+\alpha_T/T +\beta_T/\sqrt{T}}$ and $C$ is some positive constant.
Thus \begin{equation*}
\displaystyle \varlimsup_{T \to \infty}\frac{1}{T}\log \mathcal{J}_T \leq \displaystyle \lim_{T \to \infty}\frac{1}{T}\log \left(\left(\beta_T\right)^f e^{T \times \max \left(\frac{\gamma^2 }{4g};\frac{\lambda^2}{d-b}\right)}\right)= \displaystyle -f\frac{d}{2} + \max \left(\frac{\gamma^2 }{4g};\frac{\lambda^2}{d-b}\right).
\end{equation*}
We also show that $$H_T+K_T \geq C \left(\beta_T\right)^f \left(\underline{h}_{T,\lambda} \exp\left(\frac{\gamma^2T}{4g}\right) + m_{\gamma,\lambda,T} \,  \exp\left({\frac{\lambda^2 T}{4\alpha_T}}\right) \right)$$
where $C$ is still some positive constant and $\underline{h}_{T,\lambda}= \sqrt{T} \, e^{-|\lambda|-\alpha_T/T} $. It leads to \begin{equation}\varliminf_{T \to \infty}\frac{1}{T}\log \mathcal{J}_T \geq \displaystyle -f\frac{d}{2} + \max \left(\frac{\gamma^2 }{4g};\frac{\lambda^2}{d-b}\right).
\end{equation}
\end{proof}

\section{Proofs of the LDPs for the couples of simplified estimators}

\subsection{Proof of Theorem ~\ref{LDPcouple}}

We will now establish an LDP for the first couple of simplified estimators. We notice that
$(\wt{a}_T,\wt{b}_T)= f(\sqrt{X_T/T}, S_T, \Sigma_T)$ where $f$ is the function defined on $\{(x,y,z) \in \dR^3 | yz \neq 1\} $ by 
$$f(x,y,z)=\left(\frac{2zy-x^2}{yz-1},\frac{(x^2-2)z}{yz-1}\right).$$
Thus, we first compute an LDP for the triplet $\left(\sqrt{X_T/T}, S_T, \Sigma_T\right)$ and then apply the contraction principle to the obtained rate function.

\begin{lem}\label{LDPtriplet}
The sequence $\left\lbrace \left(\sqrt{X_T/T}, S_T, \Sigma_T\right)\right\rbrace$ satisfies an LDP with good rate function: 
$$I(x,y,z) = \left\lbrace \begin{array}{ll}
  \displaystyle \frac{ab}{4}+\frac{b^2}{8}y+\frac{(a-2)^2}{8}z-\frac{b}{4}x^2+\frac{(x^2+2)^2 z}{8(yz-1)} & \text{if } x\geq 0 ,y,z,yz-1>0\\
  \displaystyle +\infty & \text{otherwise. } 
\end{array}   \right.$$
\end{lem}

\begin{proof}
See appendix D.
\end{proof}

As $f$ is continuous over ${\mathcal{D}_I}= \{(x,y,z) \in \dR^3 | I(x,y,z) <+\infty \}$, we deduce from Lemma~\ref{LDPtriplet} together with the contraction principle
that $(\wt{a}_T,\wt{b}_T)$ satisfies an LDP with good rate function $J_{a,b}$ given by
\begin{equation} J_{a,b}(\alpha,\beta) = \inf_{\mathcal{D}_I} \left\lbrace I(x,y,z) \, | f(x,y,z)=(\alpha,\beta) \right\rbrace
\end{equation}
which reduces to 
$$J_{a,b}(\alpha,\beta) = \inf_{\mathcal{D}_I} \left\lbrace \frac{ab}{4}+\frac{b^2}{8}y+\frac{(a-2)^2}{8}z-\frac{b}{4}x^2+\frac{(x^2+2)^2 z}{8(yz-1)} \:
 | f(x,y,z)=(\alpha,\beta) \right\rbrace$$
where the infimum over the empty set is equal to the infinity. One easily see that $J_{a,b}(\alpha,\beta)=+\infty$ as soon as $\alpha =2$ and $\beta \neq 0$ or $\beta=0$ and $\alpha \neq 2$, since we take the infimum over the empty set.
For the remaining particular case $(\alpha,\beta)=(2,0)$, as $z>0$ on $\mathcal{D}_I$, the only way of satisfying $f(x,y,z)=(2,0)$ is to take $x^2=2$. Therefore
\begin{equation}\label{J1} J_{a,b}(2,0)=\inf_{y>0, z>0, yz-1>0} \left\lbrace\frac{ab}{4}+\frac{b^2}{8}y+ \frac{(a-2)^2}{8}z-\frac{b}{2}+\frac{2z}{yz-1}\right\rbrace=-b
\end{equation}
 Otherwise, the condition $f(x,y,z)=(\alpha,\beta)$ implies that $$z=\frac{\beta}{2-\alpha}, \:\: \: yz-1 = \frac{\beta y + \alpha-2}{2-\alpha} \: \:\text{ and } \: \: x^2=\beta y +\alpha.$$  Thus, if  $\frac{\beta}{2-\alpha}<0$ then  $J_{a,b}(\alpha,\beta)=+\infty,$ otherwise
\begin{equation*}J_{a,b}(\alpha,\beta) = \inf_{\mathcal{D}_y} \left\lbrace \frac{ab}{4}+\frac{(a-2)^2}{8}\frac{\beta}{2-\alpha}-\frac{b}{4}\alpha+(b^2-2b\beta)\frac{y}{8} + \frac{\beta}{8} \frac{(\beta y +\alpha +2)^2}{\beta y + \alpha -2} \right\rbrace
\end{equation*}
where $\mathcal{D}_y = \{y>\frac{2-\alpha}{\beta} | \beta y + \alpha \geq 0\}$. We set on $\mathcal{D}_y$, \begin{equation}g(y):= \frac{ab}{4}+\frac{(a-2)^2}{8}\frac{\beta}{2-\alpha}-\frac{b}{4}\alpha+ (b^2-2b\beta)\frac{y}{8} + \frac{\beta}{8} \frac{(\beta y +\alpha +2)^2}{\beta y + \alpha -2}.
\end{equation} 
Its derivative vanishes at point $$y_0=\frac{2-\alpha}{\beta}-\frac{4}{b-\beta}$$ if $\beta \neq b$ and at $y_0=\frac{2-\alpha}{\beta}$ if $\beta=b$.
Depending on the values of $\beta$ and $\alpha$, the infimum will be reached either at the critical point $y_0$ or at the boundary of the domain $\mathcal{D}_y$.
\\
$\bullet$ \textbf{For $\boldsymbol{\beta<0}$:} Only the case $\alpha>2$ has not been considered yet.
The condition $\beta y+\alpha \geq 0$ implies that $y\leq -\frac{\alpha}{\beta}$. Therefore,
$$J_{a,b}(\alpha,\beta)=\inf_{\frac{2-\alpha}{\beta} < y \leq -\frac{\alpha}{\beta}} g(y).$$  For $\beta \leq b$, $y_0$ in not inside the domain over which we take the infimum. Moreover $g$ tends to infinity when $y$ tends to $\frac{2-\alpha}{\beta}$, so necessarily $$J_{a,b}(\alpha,\beta)=g\left(-\frac{\alpha}{\beta}\right)=\frac{(a-2)^2\beta}{8(2-\alpha)} \left(1+\frac{(2-\alpha)b}{\beta(a-2)}\right)^2-\frac{\beta}{4}\left(1-\frac{b}{\beta}\right)^2.$$
For $\beta > b$, the condition $\frac{2-\alpha}{\beta}-\frac{4}{b-\beta}>\frac{2-\alpha}{\beta}$ is always satisfied and $\frac{2-\alpha}{\beta}-\frac{4}{b-\beta} \leq \frac{-\alpha}{\beta}$ if and only if $\beta \geq \frac{b}{3}$. Consequently, if $b<\beta \leq \frac{b}{3}$, the derivative does not vanish on the domain and we find the same value of $J_{a,b}$ as above, while if $\frac{b}{3}<\beta<0$, we get
 $$J_{a,b}(\alpha,\beta)=g\left(\frac{2-\alpha}{\beta}-\frac{4}{b-\beta}\right)= \frac{(a-2)^2\beta}{8(2-\alpha)} \left(1+\frac{(2-\alpha)b}{\beta(a-2)}\right)^2+2\beta-b.$$
 $\bullet$ \textbf{For $\boldsymbol{\beta>0}$:} the condition $\beta y +\alpha \geq 0$ becomes $y \geq -\frac{\alpha}{\beta}$ which is smaller than $\frac{2-\alpha}{\beta}$. Consequently, $$J_{a,b}(\alpha,\beta)=\inf_{\frac{2-\alpha}{\beta} < y} g(y).$$
 The derivative is equal to zero for $y=\frac{2-\alpha}{\beta}-\frac{4}{b-\beta}$ which is always greater than $\frac{2-\alpha}{\beta}$ so inside the domain. We get 
 $$J_{a,b}(\alpha,\beta)=g\left(\frac{2-\alpha}{\beta}-\frac{4}{b-\beta}\right)= \frac{(a-2)^2\beta}{8(2-\alpha)} \left(1+\frac{(2-\alpha)b}{\beta(a-2)}\right)^2+2\beta-b.$$
\demend

\subsection{Proofs of Corollaries ~\ref{LDPb} and ~\ref{LDPa}}

Using the contraction principle again, we deduce from theorem ~\ref{LDPcouple} LDPs for both estimators. We begin with $\wt{b}_T$ because the calculations are really straightforward.

\begin{proof}[Proof of Corollary ~\ref{LDPb}]
From the contraction principle, we know that $$\displaystyle J_b(\beta)=\inf_{\alpha \in \dR} I_{a,b}(\alpha,\beta).$$
We have directly that $J_b(0)=J_{a,b}(2,0)=-b$ and that, for $\beta \neq 0$, $$J_b(\beta)=J_{a,b}(2+\beta\, \frac{a-2}{b},\beta).$$ This leads to the result noticing that it is continuous at point zero.
\end{proof}

\begin{proof}[Proof of Corollary ~\ref{LDPa}]
With the contraction principle again, we have  $$J_a(\alpha)= \inf_{\beta \in \dR} J_{a,b}(\alpha,\beta).$$
$\bullet$ \textbf{For $\boldsymbol{\alpha=2}$:}
We easily show that $J_a(2)=J_{a,b}(2,0)=-b$.\\
$\bullet$ \textbf{For $\boldsymbol{\alpha<2}$:} Investigating for critical points, we obtain that
$$\displaystyle J_a(\alpha)= \inf_{\beta>0}\left\lbrace \frac{(a-2)^2\beta}{8(2-\alpha)} \left(1+\frac{(2-\alpha)b}{\beta(a-2)}\right)^2+2\beta-b \right\rbrace =J_{a,b}(\alpha,\beta_0),$$ 
where $\beta_0$ is the critical point given by $$\displaystyle \beta_0=-(2-\alpha) \,  b\left((a-2)^2+16\, (2-\alpha)\right)^{-1/2}.$$
This straightforwardly leads to the announced result
$$J_a(\alpha) = \frac{b}{4}\left(a-6-\sqrt{(a-2)^2+16\, (2-\alpha)}\right).$$
\\
$\bullet$ \textbf{For $\boldsymbol{\alpha>2}$:}
$J_{a,b}(\alpha,\beta)=\min ( I_1,I_2)$ where \begin{equation}I_1=\inf_{\beta\leq\frac{b}{3}}\left\lbrace \frac{(a-2)^2\beta}{8(2-\alpha)} \left(1+\frac{(2-\alpha)b}{\beta(a-2)}\right)^2-\frac{\beta}{4}\left(1-\frac{b}{\beta}\right)^2 \right\rbrace 
\end{equation} and
\begin{equation}I_2=\inf_{\frac{b}{3}\leq\beta<0}\left\lbrace \frac{(a-2)^2\beta}{8(2-\alpha)} \left(1+\frac{(2-\alpha)b}{\beta(a-2)}\right)^2+2\beta-b \right\rbrace  .
\end{equation}
For $I_2$, with the calculations of the second case, we already know that the derivative equals zero for $\beta_0$ satisfying $$\displaystyle \beta_0^2=\frac{(2-\alpha)^2 \,  b^2}{(a-2)^2+16\, (2-\alpha)}.$$ This is well defined if and only if $(a-2)^2+16\, (2-\alpha) >0$. And, $\beta_0$ is in the domain if and only if $\beta_0<0$ and $\beta_0^2 \leq \frac{b^2}{9}$. All those conditions are fulfilled if and only if $9(2-\alpha)^2 \leq (a-2)^2+16 (2-\alpha)$. As $\alpha >2$, we obtain the condition $$\alpha < \ell_a := \frac{10}{9}+\frac{1}{9}\sqrt{64+9(a-2)^2}$$ and  $\displaystyle \beta_0=-(2-\alpha) \,  b\left((a-2)^2+16\, (2-\alpha)\right)^{-1/2}.$ 
Thus, for $2<\alpha<\ell_a$ 
\begin{equation}\label{I21}I_2=J_{a,b}(\alpha,\beta_0)= \displaystyle \frac{b}{4} \left(a-6-\sqrt{(a-2)^2+16(2-\alpha)} \right).
\end{equation}
Otherwise, for $\alpha \geq \ell_a$, the derivative never vanishes on the domain and the minimum is reached at one of the boundaries. When $\beta$ goes to zero, the function goes to infinity. Consequently,
\begin{equation}\label{I22}I_2=J_{a,b}(\alpha,\frac{b}{3})=\frac{b}{3}\left( \frac{\left((a-2)+3(2-\alpha)\right)^2}{8(2-\alpha)}-1\right).
\end{equation}
For $I_1$, the idea is similar. 
The derivative equals zero for $\beta_1$ satisfying $$\displaystyle \beta_1^2=\frac{\alpha\, b^2 \, (\alpha-2)}{(a-2)^2+2\, (\alpha-2)}.$$ This time the domain is $\left\lbrace\beta<\frac{b}{3}\right\rbrace$. So $\beta_1$ is inside the domain if $\beta_1<0$ and $\beta_1^2 \geq \frac{b^2}{9}$. It leads us to $$\beta_1= b \, \sqrt{\alpha \, (\alpha-2)}\left((a-2)^2+2\, (\alpha-2)\right)^{-1/2}$$   with the condition $9\alpha (\alpha-2) > (a-2)^2 +2(\alpha-2)$ on $\alpha$ which gives the same limit value $\ell_a$. We get
\begin{equation}\label{I1}I_1 = \left\lbrace \begin{array}{ll}
J_{a,b}(\alpha,\frac{b}{3}) =\displaystyle \frac{b}{3}\left( \frac{\left((a-2)+3(2-\alpha)\right)^2}{8(2-\alpha)}-1\right) &\text{ if } 2<\alpha < \ell_a \\
 J_{a,b}(\alpha,\beta_1) = \displaystyle \frac{b}{4}\left(a-\sqrt{\alpha\left(\frac{(a-2)^2}{\alpha-2}+2\right)} \right) &\text{ if } \alpha \geq \ell_a.
\end{array}\right. 
\end{equation}
We now come back to $J_{a,b}(\alpha,\beta)$. Combining (\ref{I21}), (\ref{I22}) and (\ref{I1}), we obtain 
$$J_{a,b}(\alpha,\beta)= \min (I_1,I_2)= \left\lbrace \begin{array}{ll}
 \min\left(J_{a,b}(\alpha,\beta_0),J_{a,b}(\alpha,\frac{b}{3}) \right) &\text{ if } 2< \alpha \leq \ell_a \\
\min \left(J_{a,b}(\alpha,\beta_1),J_{a,b}(\alpha,\frac{b}{3})\right) &\text{ if } \alpha > \ell_a ,
\end{array}\right. $$
and it is easy to deduce that 
\begin{equation}J_{a,b}(\alpha,\beta)= \left\lbrace \begin{array}{ll}
J_{a,b}(\alpha,\beta_0)&\text{ if } 2<\alpha \leq \ell_a \\
J_{a,b}(\alpha,\beta_1)&\text{ if } \alpha > \ell_a .
\end{array}\right. 
\end{equation}
This leads to the conclusion, noticing that it is continuous at the point $\alpha=2$.
\end{proof}

\subsection{Proof of Theorem ~\ref{LDPcouple2}}

We consider the second couple of simplified estimators defined by
$$\displaystyle \wc{a}_T=\frac{S_T \left(2 \, \Sigma_T +  L_T \right)}{V_T} \, \, \text{ and } \, \, \displaystyle \wc{b}_T=\frac{-2\, \Sigma_T-L_T}{V_T}.$$
We notice that $\left(\wc{a}_T,\wc{b}_T \right)=h(S_T,\Sigma_T,\mathcal{L}_T)$, where $h$ is the function defined on $\left\lbrace(y,z,t) \in \mathbb{R}^3 | yz-1 \neq 0 \right\rbrace$ by
\begin{equation}\label{defh}
h(y,z,t)=\left(\frac{y \, (2z-t^2 \mathbf{1}_{t \leq 0}+t\mathbf{1}_{t > 0})}{yz-1},\, \frac{t^2\mathbf{1}_{t \leq 0}-t\mathbf{1}_{t>0}-2z}{yz-1}\right).
\end{equation}
Once again, we start by computing an LDP for the triplet $(S_T, \Sigma_T, \mathcal{L}_T)$ and then we deduce an LDP for the couple of estimators applying the contraction principle to the obtained rate function.

\begin{lem}\label{LDPtriplet2}
The sequence $\left\lbrace(S_T, \Sigma_T, \mathcal{L}_T)\right\rbrace$ satisfies an LDP with good rate function 
$$\widetilde{I}(y,z,t) = \left\lbrace \begin{array}{ll}
  \displaystyle \frac{ab}{4}+\frac{b^2}{8}y+\frac{(a-2)^2}{8}z+\frac{a}{4}t^2+\frac{4z(yt^2+1)+t^4 y}{8(yz-1)} & \text{if } t\leq 0, y>0, z>0\\
  & \text{ and } yz-1>0, \\
  \displaystyle +\infty & \text{otherwise. } 
\end{array}   \right.$$
\end{lem}

\begin{proof}
The proof is postponed to Appendix E.
\end{proof}

As the function $h$ given by \eqref{defh} is continuous over the  domain where the rate function $\wt{I}$ of Lemma~\ref{LDPtriplet2} is finite, the contraction principle applied to $\wt{I}$ shows that $\left(\wc{a}_T,\wc{b}_T \right)$ satisfies an LDP with good rate function $K_{a,b}$ given by
$$K_{a,b}(\alpha,\beta)= \inf_{\mathcal{D}_{\wt{I}}} \left\lbrace \wt{I}(y,z,t) | h(y,z,t)=(\alpha,\beta)\right\rbrace$$
which reduces to
$$K_{a,b}(\alpha,\beta)=\underset{\mathcal{D}_{\alpha,\beta}}{\inf} \left \lbrace \frac{ab}{4}+\frac{b^2}{8}y+\frac{(a-2)^2}{8}z+ \frac{a}{4}t^2 +\frac{4z(yt^2+1)+t^4 y}{8(yz-1)} \right\rbrace$$
where $\mathcal{D}_{\alpha,\beta}=\left\lbrace(y,z,t)\in\mathbb{R}^3 | y>0, z>0, yz-1>0, t \leq 0 \text{ and } h(y,z,t)=(\alpha,\beta)\right\rbrace$ and the infimum over the empty set is equal to infinity.
The condition $h(y,z,t)=(\alpha,\beta)$ implies that $$\beta y=-\alpha \, \, \text{ and } \, t^2=(2-\alpha) z - \beta.$$
It gives us some additional conditions on the parameters. First of all, we notice that for $\alpha$ or $\beta$ equal to zero, $\mathcal{D}_{\alpha,\beta}$ is not empty if and only if the other one is also zero. Thus, $K_{a,b}(\alpha,\beta)=+\infty$ over $\{0\}\times\dR^{+}_{*}$ and $\dR^{+}_{*}\times \{0\}$. 
If $\alpha=\beta=0$, necessarily $t^2=2z$ and we easily obtain the critical points leading to
$$K_{a,b}(0,0)=-\frac{b}{4}\left(4-a+\sqrt{a^2+16}\right) .$$
 Moreover, for $\beta\neq 0$, as $y=-\frac{\alpha}{\beta} $, $\mathcal{D}_{\alpha,\beta}$ is empty as soon as $\alpha$ and $\beta$ have the same sign. So $K_{a,b}(\alpha,\beta)=+\infty$ over $\dR^{+}_{*} \times \dR^{+}_{*}$ and $\dR^{-}_{*} \times \dR^{-}_{*}$. Besides, both expressions will give us boundaries for $z$ depending on the sign of $\alpha$, $\beta$ and $2-\alpha$, because $t^2$ must be positive and $z$ must be greater than $\frac{1}{y}$.
Assuming that all conditions are fulfilled, we derive from Lemma ~\ref{LDPtriplet2} that
$$\wt{I}\left(-\frac{\alpha}{\beta},z, \sqrt{(2-\alpha)z-\beta}\right)= \displaystyle A_{\alpha,\beta}+C_{\alpha} z-\frac{2\beta z}{\alpha z+\beta}$$
where $C_{\alpha}$ and $A_{\alpha,\beta}$ do not depend on $z$, and are defined by \begin{equation} \displaystyle C_{\alpha}=\frac{1}{8} \left(a-\alpha\right)^2+2-\alpha \, \, \text{ and } \, \, A_{\alpha,\beta}= - \frac{\alpha}{\beta} \frac{b^2}{8}+\frac{ab}{4} - \frac{a\beta}{4}+\frac{\alpha \beta}{8}.
\end{equation}
Thus, \begin{equation} \label{defK}
K_{a,b}(\alpha,\beta)= \underset{\mathcal{D}_{z}}{\inf} \left \lbrace A_{\alpha,\beta}+C_{\alpha} z-\frac{2\beta z}{\alpha z+\beta} \right\rbrace
\end{equation}
where $\mathcal{D}_{z}=\left\lbrace z>0 | z > -\frac{\beta}{\alpha} \text{ and } (2-\alpha)z-\beta \geq 0\right\rbrace$. Depending on the values of $\alpha$ and $\beta$, the infimum will be reached either at a critical point or at the boundary of the domain $\cD_z$.\\
$\bullet$ \textbf{For $\boldsymbol{\alpha<0}$:} Only remains the case $\beta>0$. As $0<\frac{\beta}{2-\alpha}\leq \frac{\beta}{-\alpha}$ the domain $\mathcal{D}_z$ reduces to $\mathcal{D}_z=\left\lbrace z > -\frac{\beta}{\alpha} \right\rbrace$. We look for critical points of $A_{\alpha,\beta}+C_{\alpha} z-\frac{2\beta z}{\alpha z+\beta}$ over this domain. We find that critical points $z_0$ satisfy 
$$\left(\alpha z_0+\beta\right)^2=\frac{2\beta^2}{C_{\alpha}}$$
We notice that for $\alpha$ negative $C_{\alpha}$ is always positive. So the only critical point that remains in the domain is 
 $z_0=-\frac{\beta}{\alpha}\left(1 + \sqrt{\frac{2}{C_{\alpha}}}\right)$. As the function tends to infinity on the boundaries of $\mathcal{D}_z$, it actually reaches the infimum we were looking for at this critical point $z_0$. Replacing it into (\ref{defK}), we find \begin{equation}\label{infK1}
 \begin{array}{lcl}
 K_{a,b}(\alpha,\beta) &= & \displaystyle A_{\alpha,\beta}+C_{\alpha} z_0-\frac{2\beta z_0}{\alpha z_0+\beta}\\
 &=&
\displaystyle \frac{a}{4}\left(b-\beta\right)-\frac{\alpha}{8\beta} \left(b^2-\beta^2\right)- \frac{\beta}{\alpha}\left(\sqrt{2}+\sqrt{C_{\alpha}}\right)^2 .
\end{array}
\end{equation}
\\ 
$\bullet$ \textbf{For $\boldsymbol{0<\alpha\leq 2}$:} As $\beta<0$, the condition $(2-\alpha)z-\beta \geq 0$ is always verified. Thus $\mathcal{D}_z=\left\lbrace z > -\frac{\beta}{\alpha}\right\rbrace$. We obtain the same critical point than in the firsta case and the infimum is still given by formula (\ref{infK1}).
\\
$\bullet$ \textbf{For $\boldsymbol{\alpha>2}$.} The case $\beta>0$ has already been seen. We consider $\beta<0$. We investigate the critical points of $A_{\alpha,\beta}+C_{\alpha} z-\frac{2\beta z}{\alpha z+\beta}$ over the domain $\mathcal{D}_z$, given in this case by $\mathcal{D}_z=\left\lbrace -\frac{\beta}{\alpha}<z \leq \frac{\beta}{2-\alpha}\right\rbrace$. We need to distinguish cases depending on the sign of $C_{\alpha}$.

If $\alpha < a+4-2\sqrt{a}$ (which is greater than $2$ because $a>2$), $C_{\alpha}$ is positive and we find the same critical points than in the first case. The condition $z > -\frac{\beta}{\alpha}$ is still only verified by $$z_0=-\frac{\beta}{\alpha}\left(1 + \sqrt{\frac{2}{C_{\alpha}}}\right).$$ But, the condition $z_0 \leq \frac{\beta}{2-\alpha}$ is not satisfied for all $\alpha$ in $[2,a+4-2\sqrt{a}[$. Indeed, \begin{equation}\label{CondAlpha}
z_0 > \frac{\beta}{2-\alpha} \text{ if and only if }\frac{2}{C_{\alpha}}(\alpha-2)^2 >4
\end{equation} which leads to the following condition on $\alpha$: $$3\alpha^2+2(a-4)\alpha-(a^2+8) >0.$$ One of the roots is negative. The other one is inside $[2,a+4-2\sqrt{a}[$: \begin{equation}\label{alphaa}
\alpha_a=-\frac{2}{3}\left(\frac{a}{2}-2-\sqrt{a^2-2a+4} \right).
\end{equation} Thus, for $\alpha \leq \alpha_a$, the infimum is reached at $z_0$ and is given by (\ref{infK1}), while for $\alpha>\alpha_a$, $z_0$ is not inside the domain $\mathcal{D}_z$ so the derivative does not vanish and the infimum is reached at one of the boundaries. We notice that for $z$ tending to $-\frac{\beta}{\alpha}$, $A_{\alpha,\beta}+C_{\alpha} z-\frac{2\beta z}{\alpha z+\beta}$ tends to the infinity. Thus, the infimum is reached at the other boundary of the domain: $z_0=\frac{\beta}{2-\alpha}$. Replacing it into (\ref{defK}), we obtain
\begin{equation}\label{infK2}
K_{a,b}(\alpha,\beta)= \frac{a}{4}\left(b-\beta\right)-\frac{\alpha}{8\beta} \left(b^2-\beta^2\right) - \frac{\beta \left(a-\alpha\right)^2}{8 \left(\alpha-2\right)}.
\end{equation}

If $\alpha \in [a+4-2\sqrt{a};a+4+2\sqrt{a}]$, then $C_{\alpha}$ is null or negative so the derivative cannot vanish and $K_{a,b}$ is given by (\ref{infK2}). 

If $\alpha > a+4+2\sqrt{a}$, then $C_{\alpha}$ is positive but as $\alpha>\alpha_a$ the critical point $z_0$ is greater than $\frac{\beta}{2-\alpha}$  so outside the domain $\mathcal{D}_z$. The infimum is reached at the boundary $\frac{\beta}{2-\alpha}$ of the domain and is still given by (\ref{infK2}).
\demend

\subsection{Proofs of Corollaries ~\ref{LDPa2} and ~\ref{LDPb2}}

\begin{proof}[Proof of Corollary ~\ref{LDPb2}]
The result is a direct application of the contraction principle to the rate function $K_{a,b}$ of Theorem~\ref{LDPcouple2}. We did not obtain an explicit expression of the infimum.
\end{proof}

\begin{proof}[Proof of Corollary ~\ref{LDPa2}]
With the contraction principle again, we know that $(\wc{a}_T)$ satisfies an LDP with good rate function $$K_a(\alpha)= \inf_{\beta \in \dR} K_{a,b}(\alpha,\beta).$$\\
$\bullet$ \textbf{For $\boldsymbol{\alpha=0}$.}
We easily show that $K_a(0)=K_{a,b}(0,0)$. \\
$\bullet$ \textbf{For $\boldsymbol{\alpha<0}$.} We rewrite
$$\displaystyle K_a(\alpha)= \inf_{\beta>0}\left\lbrace A_{\alpha,\beta}-\frac{\beta}{\alpha}\left(\sqrt{2}+\sqrt{C_{\alpha}}\right)^2 \right\rbrace.$$ The critical points $\beta_b$ 
satisfy \begin{equation}\label{CondBeta}
 \beta_b^2=\frac{\alpha^2 b^2}{16\sqrt{2}\sqrt{C_{\alpha}}+a^2-8\alpha+32},
 \end{equation} which is clearly well defined for all $\alpha<0$. Using the fact that $\alpha<0$ and $b<0$ and as $\beta_b$ must be positive, we obtain \begin{equation}\label{betazero}
 \beta_b= b\alpha \left(16\sqrt{2}\sqrt{C_{\alpha}}+a^2-8\alpha+32\right)^{-1/2}
 \end{equation}
 and \begin{equation}\label{Ka1} K_a(\alpha)=K_{a,b}(\alpha,\beta_b).
 \end{equation}\\
$\bullet$ \textbf{For $\boldsymbol{0<\alpha\leq \alpha_a}$.} This time $$\displaystyle K_a(\alpha)= \inf_{\beta<0}\left\lbrace A_{\alpha,\beta}-\frac{\beta}{\alpha}\left(\sqrt{2}+\sqrt{C_{\alpha}}\right)^2 \right\rbrace.$$ The critical points $\beta_b$ are still given by (\ref{CondBeta}). It is well defined for all $\alpha < \alpha_a$. Namely, $16\sqrt{2}\sqrt{C_{\alpha}}+a^2-8\alpha+32$ is a decreasing function on $\alpha$ over this domain and is positive at the point $\alpha_a$. Indeed,
we know from (\ref{CondAlpha}) and (\ref{alphaa}) that $\alpha_a>2$ satisfies $(\alpha_a-2)^2=2 C_{\alpha_a}$, which leads to $16\sqrt{2}\sqrt{C_{\alpha_a}}+a^2-8\alpha_a+32=8\alpha_a+a^2 >0$.
This time $\alpha>0$ and $\beta_b$ must be negative but we obtain anyway the same critical point $\beta_b$ given by (\ref{betazero}) and the infimum $K_a$ by (\ref{Ka1}).\\
$\bullet$ \textbf{For $\boldsymbol{\alpha> \alpha_a}$.} This time $$\displaystyle K_a(\alpha)= \inf_{\beta<0}\left\lbrace \frac{(a-2)^2\beta}{8(2-\alpha)} \left(1+\frac{(2-\alpha)b}{\beta(a-2)}\right)^2-\frac{\beta}{4}\left(1-\frac{b}{\beta}\right)^2 \right\rbrace= J_a(\alpha).$$ Using the results of the third case of the proof of Corollary ~\ref{LDPa}, we obtain the same critical point and the same infimum.
\end{proof}

\section{Proof of Theorem~\ref{LDP_MLE}}

We now come back to the MLE $\left(\wh{a}_T,\wh{b}_T\right)$. As we did for the couples of simplified estimators, we first establish an LDP for the quadruplet $\mathcal{Q}_T$ and we deduce an LDP for the MLE via the contration principle.

\subsection{Existence of an LDP}

\begin{lem}\label{LDP_quad}
The quadruplet $\mathcal{Q}_T=(\sqrt{X_T/T},S_T,\Sigma_T,\mathcal{L}_T)$ satisfies an LDP with good rate function $\Lambda^{*}$ given by
\begin{equation}\label{domL} \Lambda^{*}(x,y,z,t)=+\infty \text{ for }  x<0, t>0, y \leq 0, z \leq 0 \text{ or } yz-1\leq 0.
\end{equation}
and, otherwise,
\begin{equation}\Lambda^{*}(x,y,z,t)=\underset{\mathcal{D}_{d,f}}{\sup} \, h(d,f)
\end{equation}
where $\mathcal{D}_{d,f}=\left\lbrace d>0,f>0 \right\rbrace$ and, with $\varphi(f)=2f+a+2$, \begin{equation*}h(d,f)=\frac{1}{4}\left(t\sqrt{\varphi(f)}-x\sqrt{d-b}\right)^2+y\, \frac{b^2-d^2}{8} +\frac{\left(a-2\right)^2-4f^2}{8} \,  z+\frac{d}{2}\left(1+f\right)+ \frac{ab}{4}.
\end{equation*}
\end{lem}

\begin{proof}
Using G\"artner-Ellis theorem, we have to compute the Fenchel-Legendre transform $\Lambda^{*}$ of the cumulant generating function $\Lambda$ defined in Proposition ~\ref{CGFquadruplet}:
\begin{equation}\Lambda^{*}(x,y,z,t)= \underset{\mathcal{D}}{\sup}\left\lbrace x\lambda+y\mu+z\nu+t\gamma-\Lambda(x,y,z,t)\right\rbrace
\end{equation}
where $\mathcal{D}=\left\lbrace \lambda \in \dR, \, \gamma \in \mathbb{R}, \, \mu<\frac{b^2}{8}, \, \nu < \frac{(a-2)^2}{8} \right\rbrace$. We show with the same arguments than for the other LDP proofs that \begin{equation} \Lambda^{*}(x,y,z,t)=+\infty \text{ for }  x<0, t>0, y \leq 0, z \leq 0 \text{ or } yz-1\leq 0.
\end{equation}
Besides, for $x \geq 0$, the part involving $\lambda$ in the function we want to optimize is always negative for $\lambda\leq 0$ and sometimes positive for $\lambda> 0$. Thus the supremum is necessarily reached for some $\lambda>0$. With the same argument for $t \leq 0$, we show that we only have to consider $\gamma<0$. Replacing $\mu$ and $\nu$ by their expression in $d$ and $f$, the domain $\mathcal{D}$ over which we optimize reduces to $\mathcal{D}=\left\lbrace \lambda>0, \gamma<0,d>0,f>0\right\rbrace$. 
Replacing $\Lambda$ by its value leads to $$\Lambda^{*}(x,y,z,t)=\max (S_1,S_2),$$ where $$S_1=\underset{\mathcal{D} \cap \left\lbrace \frac{\gamma^2}{\lambda^2} \leq \frac{\varphi(f)}{d-b}\right\rbrace}{\sup} \left\lbrace x\lambda+y\frac{b^2-d^2}{8}+z \frac{(a-2)^2-4f^2}{8}+t\gamma+\frac{d}{2}(1+f)+\frac{ab}{4} -\frac{\lambda^2}{d-b}\right\rbrace $$ and
$$S_2=  \underset{\mathcal{D} \cap \left\lbrace \frac{\gamma^2}{\lambda^2} \geq \frac{\varphi(f)}{d-b}\right\rbrace}{\sup} \left\lbrace x\lambda+y\frac{b^2-d^2}{8}+z \frac{(a-2)^2-4f^2}{8}+t\gamma+\frac{d}{2}(1+f)+\frac{ab}{4}-\frac{\gamma^2}{\varphi(f)}\right\rbrace.$$
We first consider $S_1$. The domain over which we take the supremum is given by $$\mathcal{D} \cap \left\lbrace \frac{\gamma^2}{\lambda^2} \leq \frac{\varphi(f)}{d-b}\right\rbrace=\left\lbrace \lambda<0,d>0,f>0,0> \gamma \geq -\sqrt{\frac{\varphi(f)}{d-b}} \lambda\right\rbrace.$$
Over this domain, as $t\leq 0$, $0 \leq t\gamma \leq -t\sqrt{\frac{\varphi(f)}{d-b}} \lambda$, so that the supremum of $t\gamma$ is equal to $-t\lambda\sqrt{\frac{\varphi(f)}{d-b}} $. Thus, if we set $\mathcal{D}_1=\left\lbrace \lambda<0,d>0,f>0\right\rbrace $,
$$S_1=\underset{\mathcal{D}_1}{\sup}\left\lbrace x\lambda+y\frac{b^2-d^2}{8}+z \frac{(a-2)^2-4f^2}{8}-t\lambda\sqrt{\frac{\varphi(f)}{d-b}}+\frac{d}{2}(1+f)+\frac{ab}{4} -\frac{\lambda^2}{d-b}\right\rbrace .$$ 
The supremum over $\lambda$ is easy to compute. Indeed, the function is concave on $\lambda$ and the critical point is given by 
$$\lambda=\frac{d-b}{2} \left(x-t\sqrt{\frac{\varphi(f)}{d-b}}\right).$$
Finally, with $\mathcal{D}_{d,f}=\left\lbrace d>0,f>0 \right\rbrace$, we obtain 
$$S_1= \underset{\mathcal{D}_{d,f}}{\sup} \left\lbrace \frac{1}{4}\left(t\sqrt{\varphi(f)}-x\sqrt{d-b}\right)^2+y\frac{b^2-d^2}{8} +\frac{\left(a-2\right)^2-4f^2}{8} \,  z+\frac{d}{2}\left(1+f\right)+ \frac{ab}{4}\right\rbrace.$$
We do the same thing with $S_2$, computing first the supremum over $\lambda$ and then over $\gamma$. We obtain $S_1=S_2$, so that
\begin{equation}\Lambda^{*}(x,y,z,t)=S_1=\underset{\mathcal{D}_{d,f}}{\sup} \, h(d,f)
\end{equation}
where \begin{equation*}h(d,f)=\frac{1}{4}\left(t\sqrt{\varphi(f)}-x\sqrt{d-b}\right)^2+y\, \frac{b^2-d^2}{8} +\frac{\left(a-2\right)^2-4f^2}{8} \,  z+\frac{d}{2}\left(1+f\right)+ \frac{ab}{4}.
\end{equation*} 
\end{proof}

\begin{rem}\label{rem1}
This supremum is not explicitly computable but, as the function $h$ is concave, it is reached for some $\left(d^{*},f^{*}\right)$ and this gives the rate function of the LDP satisfied by the quadruplet $\mathcal{Q}_T$.
\end{rem}

\begin{lem}\label{exist_LDP_MLE}
The couple $(\wh{a}_T,\wh{b}_T)$ satisfies an LDP with good rate function $I_{a,b}$ given over $\dR^2$ by

$$I_{a,b}(\alpha,\beta)=\left\lbrace \begin{array}{ll} 
\vspace{2ex}
K_{a,b}(0,0) &\text{if } (\alpha,\beta)=(0,0),\\
\vspace{2ex}
J_{a,b}(2,0) &\text{if } (\alpha,\beta)=(2,0),\\
\underset{\mathcal{D}_{x,t}}{\inf} \, \,  \underset{\mathcal{D}_{d,f}}{\sup} \, H(x,t,d,f)&\text{if } (\alpha,\beta) \in \mathcal{D}_1 \cup \mathcal{D}_2 \cup \mathcal{D}_3, \\
+\infty  & \text{otherwise.} \\
\end{array}\right.$$ 
where $\mathcal{D}_1= \dR^{-}\times \dR^{+}_{*}$, $\mathcal{D}_2=]0,2[\times \dR  $, $\mathcal{D}_3=[2,+\infty[ \times \dR^{-}_{*}$, $\mathcal{D}_{d,f}=\left\lbrace d>0, f>0 \right\rbrace$,  $$\mathcal{D}_{x,t}=\left\lbrace x \geq 0, t \leq 0 | \frac{x^2-\alpha}{\beta} \frac{t^2+\beta}{2-\alpha} >1 \right\rbrace $$ and $$\begin{array}{lcl}
H(x,t,d,f)&=&\displaystyle \frac{1}{4}\left(t\sqrt{2f+a+2}-x\sqrt{d-b}\right)^2+\frac{b^2-d^2}{8} \frac{x^2-\alpha}{\beta}\\
& &\displaystyle +\frac{\left(a-2\right)^2-4f^2}{8} \,  \frac{t^2+\beta}{2-\alpha}+\frac{d}{2}\left(1+f\right)+ \frac{ab}{4}.
\end{array} $$
\end{lem}

\begin{proof}
$(\wh{a}_T,\wh{b}_T)=g(\mathcal{Q}_T)$
where $g$ is the function defined on $\left\lbrace(x,y,z,t) \in \mathbb{R}^4 | yz-1 \neq 0 \right\rbrace$ by
$$g(x,y,z,t)=\left(\frac{y \, (2z-t^2\mathbf{1}_{t\leq 0}+t\mathbf{1}_{t>0})-x^2}{yz-1},\, \frac{t^2\mathbf{1}_{t\leq 0}-t\mathbf{1}_{t>0}+(x^2-2)z}{yz-1}\right).$$ As $g$ is continuous over the domain $\cD_{\Lambda^{*}}$  where the rate function $\Lambda^{*}$ of Lemma~\ref{LDP_quad} is finite, the contraction principle applies and give us that the couple $(\wh{a}_T,\wh{b}_T)$ satisfies an LDP with good rate function $I_{a,b}$ given by \begin{equation}I_{a,b}(\alpha,\beta)= \inf_{\mathcal{D}_{\Lambda^{*}}} \left\lbrace \Lambda^{*}(x,y,z,t) | g(x,y,z,t)=(\alpha,\beta) \right\rbrace.
\end{equation}
The condition $g(x,y,z,t)=(\alpha,\beta)$ gives us a link between $x$ and $y$ and one between $t$ and $z$:
\begin{equation}\label{cond}\beta y=x^2-\alpha \,\, \, \text{ and } \, \, \,\left(2-\alpha\right) z=t^2+\beta.
\end{equation}
We first notice that if $\alpha\leq0$ and $\beta<0$ then $y$ is negative and for all $x$, $z$, $t$, $\Lambda^{*}(x,y,z,t)=+\infty$ such as $I_{a,b}(\alpha,\beta)$. Similarly, if $\beta\geq0$ and $\alpha >2$, $z$ is negative and $\Lambda^{*}(x,y,z,t)=+\infty$ for all $x$, $y$, $t$, then $I_{a,b}(\alpha,\beta)=+\infty$. If $\beta=0$ and $\alpha<0$, the first condition in (\ref{cond}) leads to $x^2$ negative and if $\alpha=2$ and $\beta>0$ the second condition gives $t^2$ negative. So, in both cases, we get $I_{a,b}(\alpha,\beta)=+\infty$.
 We now focus on the values of $\alpha$ and $\beta$ for which $I_{a,b}$ is not clearly infinite.
 We first consider the two remaining limit cases : $(0,0)$ and $(2,0)$. If $\alpha=\beta=0$ then the first condition of (\ref{cond}) gives $x^2=\alpha=0$ so that 
 \begin{equation}I_{a,b}(0,0)=K_{a,b}(0,0).
 \end{equation} 
 Similarly, if $\alpha=2$ and $\beta=0$, the second condition implies that $t^2=(2-\alpha)z=0$ and consequently \begin{equation}I_{a,b}(2,0)=J_{a,b}(2,0).
 \end{equation} For all remaining values of $(\alpha,\beta)$,
we define the function
\begin{equation}\begin{array}{lcl}
H(x,t,d,f)&=&\displaystyle \frac{1}{4}\left(t\sqrt{2f+a+2}-x\sqrt{d-b}\right)^2+\frac{b^2-d^2}{8} \frac{x^2-\alpha}{\beta}\\
& &\displaystyle +\frac{\left(a-2\right)^2-4f^2}{8} \,  \frac{t^2+\beta}{2-\alpha}+\frac{d}{2}\left(1+f\right)+ \frac{ab}{4}
\end{array} 
\end{equation}
and obtain the announced result: \begin{equation}I_{a,b}(\alpha,\beta)=\underset{\mathcal{D}_{x,t}}{\inf} \, \,  \underset{\mathcal{D}_{d,f}}{\sup} \, H(x,t,d,f),
\end{equation} where $\mathcal{D}_{x,t}=\left\lbrace x \geq 0, t \leq 0 | \frac{x^2-\alpha}{\beta} \frac{t^2+\beta}{2-\alpha} >1 \right\rbrace$ and $\mathcal{D}_{d,f}=\left\lbrace d>0, f>0 \right\rbrace$.
\end{proof}

We were not able to compute $I_{a,b}$ explicitly at this stage. It is the aim of the next subsection. 

\subsection{Evaluating the rate function $I_{a,b}$} 
Our goal is to show that \begin{equation}I_{a,b}(\alpha,\beta)= \min \left(J_{a,b}(\alpha,\beta), K_{a,b}(\alpha,\beta) \right)
\end{equation} where $J_{a,b}$ and $K_{a,b}$ are the rate functions for the two couples of simplified estimators (see Theorems~\ref{LDPcouple} and~\ref{LDPcouple2}) and $I_{a,b}$ is given by Lemma~\ref{exist_LDP_MLE}.
We notice that 
$$K_{a,b}(\alpha,\beta)= \underset{t \leq 0}{\inf} \, \,  \underset{\mathcal{D}_{d,f}}{\sup} \, H(0,t,d,f) \,\, \,  \text{ and } \, \, \, J_{a,b}(\alpha,\beta)=\underset{x \geq 0}{\inf} \, \,  \underset{\mathcal{D}_{d,f}}{\sup} \, H(x,0,d,f).$$ 
Thus it easily follows that \begin{equation}\label{Isup}I_{a,b}(\alpha,\beta) \leq \min \left(J_{a,b}(\alpha,\beta), K_{a,b}(\alpha,\beta) \right).
\end{equation} So, we just have to show the inequality in the other side.
We denote $\wh{\theta}_T=\left(\wh{a}_T, \wh{b}_T\right)$ and $\overline{\theta}_T=\left(\wt{a}_T, \wt{b}_T\right) \mathbf{1}_{X_T\geq1} + \left(\widecheck{a}_T, \widecheck{b}_T\right) \mathbf{1}_{X_T < 1}$.

\begin{lem}\label{ExpEq}
The estimators $\overline{\theta}_T$ and $\wh{\theta}_T$ are exponentially equivalent, which means that for all $\varepsilon >0$,
\begin{equation*}
\limsup_{T \rightarrow +\infty} \frac{1}{T} \log \dP\left( \parallel \wh{\theta}_T-\overline{\theta}_T \parallel > \varepsilon \right)=-\infty.
\end{equation*}
In particular, as the sequence $(\wh{\theta}_T)$ satisfies an LDP with good rate function $I_{a,b}$, then  the same LDP holds true for $(\overline{\theta}_T)$. 
\end{lem}

\begin{proof}
From the definition of each estimator, we get that
$$\displaystyle \wh{a}_T-\left(\wt{a}_T\mathbf{1}_{X_T<1}+\widecheck{a}_T\mathbf{1}_{X_T \geq 1} \right)=\frac{S_T\, L_T\mathbf{1}_{X_T \geq 1}-\frac{X_T}{T} \mathbf{1}_{X_T<1}}{V_T}$$ and $$  \wh{b}_T-\left(\wt{b}_T \mathbf{1}_{X_T<1}+\widecheck{b}_T \mathbf{1}_{X_T \geq 1} \right)=\frac{\frac{X_T}{T}\Sigma_T \mathbf{1}_{X_T<1}-L_T \mathbf{1}_{X_T \geq 1}}{V_T}\, .$$  Thus, for all $\varepsilon >0$, $$\dP\left( \parallel \wh{\theta}_T-\overline{\theta}_T \parallel > \varepsilon \right) \leq P_T^{\varepsilon} + Q_T^{\varepsilon}+p_T^{\varepsilon}+q_T^{\varepsilon}$$ where $P_T^{\varepsilon} = \displaystyle \dP\left( \abs{\frac{S_T\, L_T\mathbf{1}_{X_T \geq 1}}{V_T}} \geq \frac{\varepsilon}{2\sqrt{2}} \right)$, $Q_T^{\varepsilon} = \displaystyle \dP\left( \abs{\frac{\frac{X_T}{T} \mathbf{1}_{X_T<1}}{V_T}} \geq \frac{\varepsilon}{2\sqrt{2}} \right)$, \\
$p_T^{\varepsilon}=\displaystyle \dP\left( \abs{\frac{\frac{X_T}{T}\Sigma_T \mathbf{1}_{X_T<1}}{V_T}} \geq \frac{\varepsilon}{2\sqrt{2}}  \right)$  and $q_T^{\varepsilon}=\displaystyle \dP\left( \abs{\frac{L_T\mathbf{1}_{X_T \geq 1}}{V_T}} \geq \frac{\varepsilon}{2\sqrt{2}}  \right)$ .
For all $\eta >0$, we have the following upper bounds:

\begin{align}\label{B1}
 P_T^{\varepsilon} & \leq \dP\left(\abs{S_T} \geq \frac{\varepsilon}{2\eta \sqrt{2}}\right) + \dP \left(\frac{\abs{L_T\mathbf{1}_{X_T \geq 1}}}{V_T} \geq \eta \right) \notag \\
 & \leq \dP\left(S_T \geq \frac{\varepsilon}{2\eta \sqrt{2}}\right) + \dP \left(L_T\mathbf{1}_{X_T \geq 1} \geq \eta^2 \right) + \dP \left(V_T \leq \eta \right)\, ,
\end{align}

\begin{equation}
p_T^{\varepsilon} \leq \dP\left(\Sigma_T \geq \frac{\varepsilon}{2\eta \sqrt{2}}\right) + \dP \left(\frac{X_T}{T} \mathbf{1}_{X_T < 1} \geq \eta^2 \right) + \dP \left(V_T \leq \eta \right) \, ,
\end{equation}

\begin{equation}
q_T^{\varepsilon} \leq \dP\left(L_T \mathbf{1}_{X_T \geq 1} \geq \frac{\varepsilon \eta}{2 \sqrt{2}} \right) + \dP \left( V_T \leq \eta \right) \, ,
\end{equation}
and
\begin{equation}
\label{B4}
Q_T^{\varepsilon} \leq \dP\left(\frac{X_T}{T} \mathbf{1}_{X_T < 1} \geq \frac{\varepsilon \eta }{2\sqrt{2}} \right) + \dP \left( V_T \leq \eta \right) \, .
\end{equation}
First of all, using Theorem ~\ref{LDPmoyennes}, we show that for all $c> -\frac{a}{b}$, 
\begin{equation}\label{L1} 
\lim_{T \rightarrow +\infty} \frac{1}{T} \log\dP \left( S_T \geq c \right) = -I(c) \end{equation}
and for all $c>-\frac{b}{a-2}$,
\begin{equation} \lim_{T \rightarrow +\infty} \frac{1}{T} \log \dP \left( \Sigma_T \geq c \right)= -J(c)
\end{equation}
where $I$ and $J$ are given in Theorem ~\ref{LDPmoyennes}. 
Likewise, we deduce from Theorem ~\ref{LDPdenominateur} that for any $c>0$ small enough 
\begin{equation}\lim_{T \rightarrow +\infty} \frac{1}{T} \log \dP \left( V_T \leq c \right) = - K\left(c\right) .
\end{equation} 
We now consider the parts involving $L_T$. For all $c>0$ and $\lambda>0$:
\begin{align*}
\dP\left(L_T \mathbf{1}_{X_T \geq 1} \geq c \right) & = \dP \left( \log X_T \mathbf{1}_{X_T \geq 1}\geq cT \right)  \\
& \leq \dE \left[ e^{\lambda \log X_T} \right] e^{-\lambda cT} \, .
\end{align*}
Hence 
$$ \frac{1}{T} \log  \dP\left(L_T\mathbf{1}_{X_T \geq 1} \geq c \right) \leq -\lambda c + \frac{1}{T} \log \left(\dE\left[ X_T^{\lambda}\right]\right).$$
Asymptotic properties of the moments of the process $X_T$ as $T$ tends to infinity can be found in Proposition 3 of \cite{KAB2}, and give that the second term tends to zero for $T$ going to infinity. Thus, for any $\lambda >0$ and $c>0$, we have the following upper bound
$$\limsup_{T \rightarrow +\infty} \frac{1}{T} \log \dP\left(L_T \mathbf{1}_{X_T \geq 1} \geq c \right) \leq -\lambda c .$$
Consequently, letting $\lambda$ go to infinity, we obtain that for all $c >0$, 
\begin{equation}\label{L4}\limsup_{T \rightarrow +\infty} \frac{1}{T} \log \dP\left(L_T\mathbf{1}_{X_T \geq 1} \geq c \right) =-\infty .
\end{equation}
Finally, we consider the terms involving $\frac{X_T}{T}$. For all $c>0$ and $\lambda>0$:
\begin{align*}
\dP\left(\frac{X_T}{T} \mathbf{1}_{X_T < 1} \geq c \right) 
& \leq \dE \left[ e^{\lambda X_T\mathbf{1}_{X_T < 1}} \right] e^{-\lambda cT} \\
& \leq  e^{\lambda-\lambda cT} \, .
\end{align*}
Hence 
$$ \frac{1}{T} \log  \dP\left(\frac{X_T}{T}\mathbf{1}_{X_T < 1} \geq c \right) \leq -\lambda c + \frac{\lambda}{T} .$$
Thus, for any $\lambda >0$ and $c>0$, we have the following upper bound
$$\limsup_{T \rightarrow +\infty} \frac{1}{T} \log \dP\left(\frac{X_T}{T} \mathbf{1}_{X_T < 1} \geq c \right) \leq -\lambda c .$$
Consequently, letting $\lambda$ go to infinity, we obtain that for all $c >0$, 
\begin{equation}\label{L5}\limsup_{T \rightarrow +\infty} \frac{1}{T} \log \dP\left(\frac{X_T}{T}\mathbf{1}_{X_T <1} \geq c \right) =-\infty .
\end{equation}
Consequently, combining the limits (\ref{L1}) to (\ref{L5}) , we are able to compute the asymptotic behaviour of the bounds (\ref{B1}) to (\ref{B4}) and we show that for all $\varepsilon >0$ and all $\eta>0$ small enough,  
$$\limsup_{T \rightarrow +\infty} \frac{1}{T} \log \dP\left( \parallel \wh{\theta}_T-\overline{\theta}_T \parallel > \varepsilon \right) \leq -M_{\varepsilon,\eta} \, , $$
where $M_{\varepsilon,\eta} =  \displaystyle \min \left\lbrace I\left(\frac{\varepsilon}{2\eta \sqrt{2}}\right), J\left(\frac{\varepsilon}{2\eta \sqrt{2}}\right), K\left(\eta\right)\right\rbrace$.
Each term in this minimum tends to infinity as $\eta$ goes to zero, so that $M_{\varepsilon,\eta} $ itself tends to infinity. This gives the announced result.
\end{proof}

\begin{proof}[Proof of Theorem~\ref{LDP_MLE}]
We have already shown in Lemma~\ref{exist_LDP_MLE} that $I_{a,b}(2,0)=J_{a,b}(2,0)$ and $I_{a,b}(0,0)=K_{a,b}(0,0)$ and that, except at this two points, $I_{a,b}$ is infinite over $\dR^{-}\times \dR^{-}$ and over $\left[2,+\infty\right[ \times \dR^{+}$.
We also know by (\ref{Isup}) that 
$$I_{a,b}(\alpha,\beta) \leq  \min \left( J_{a,b}(\alpha,\beta); K_{a,b}(\alpha,\beta) \right),$$
so we still need to establish the other inequality over the remaining domain.
In the sequel, we show that, for all compact subsets $C \subset \dR^{2}$,
\begin{equation}\label{IborneSup}\limsup_{T \rightarrow +\infty} \frac{1}{T} \log \dP\left(\overline{\theta}_T \in C \right) \leq - \inf_{(\alpha,\beta) \in C} \min \{J_{a,b}(\alpha,\beta), K_{a,b}(\alpha,\beta) \}.
\end{equation}
It is sufficient to consider compact subsets of $\dR^{2}$ instead of closed ones, as we already know that the sequence $\overline{\theta}_T$ satisfies an LDP with good rate function $I_{a,b}$ and $\dR^{2}$ is locally compact so that the family $\left(\dP\left(\overline{\theta}_T \in \bullet \right)\right)_T$ is exponentially tight (see Lemma 1.2.18 and Exercise 1.2.19 of \cite{DeZ}). This will prove the announced result as, by Lemma~\ref{ExpEq}, the sequences $\left(\overline{\theta}_T\right)_T$ and $\left(\wh{\theta}_T\right)_T$ share the same LDP.\\
First of all, we notice that $\overline{\theta}_T=g(\mathcal{Q}_T)$
where $g$ is the function defined over $\left\lbrace(x,y,z,t) \in \mathbb{R}^4 | yz-1 \neq 0 \right\rbrace$ by
\begin{equation}\label{def_g}g(x,y,z,t)=\left(\frac{y \, (2z-t^2\mathbf{1}_{t< 0})-x^2 \mathbf{1}_{t \geq 0}}{yz-1},\, \frac{t^2\mathbf{1}_{t<0}+(x^2 \mathbf{1}_{t \geq 0}-2)z}{yz-1}\right),
\end{equation}
and the quadruplet $\mathcal{Q}_T$ satisfies an LDP with good rate function $\Lambda^{*}$ given by Lemma~\ref{LDP_quad}. 
As the function $g$ given by (\ref{def_g}) is not continuous, we cannot apply directly the contraction principle. However,
\begin{equation}\label{ldp_q}
\limsup_{T \rightarrow +\infty} \frac{1}{T} \log \dP\left(\overline{\theta}_T \in C \right) \leq - \inf_{\overline{g^{-1}\left(C\right)}} \Lambda^{*}.
\end{equation}
We need to describe the subset $\overline{g^{-1}\left(C\right)}$. A quadruplet $\left(x,y,z,t\right)$ of $\dR^{4}$ belongs to $\overline{g^{-1}\left(C\right)}$ if and only if there exists a sequence $\left(x_n,y_n,z_n,t_n\right)_n$ and a sequence $\left(\alpha_n,\beta_n\right)_n \in C$ such that, as $n$ tends to infinity,
\begin{equation}
\left(x_n,y_n,z_n,t_n\right) \rightarrow \left(x,y,z,t\right)
\end{equation}
and for all $n$
\begin{equation}\label{eqg}
g\left(x_n,y_n,z_n,t_n\right)=\left(\alpha_n,\beta_n\right)
\end{equation}
As $C$ is a compact subset, up to a subsequence, there exists $\left(\alpha,\beta\right)\in C$ such that $\left(\alpha_n,\beta_n\right)$ converges to $\left(\alpha,\beta\right)$ as $n$ goes to infinity. Moreover, (\ref{eqg}) is equivalent to the following conditions for all $n$:
$$\beta_n y_n=x_n^2 \mathbf{1}_{t_n \geq 0}-\alpha_n$$ and  $$\left(2-\alpha_n\right) z_n=t_n^2\mathbf{1}_{t_n< 0}+\beta_n.$$
Up to a subsequence again, both indicator functions converge toward $1$ or $0$. Thus, letting $n$ go to infinity, we obtain conditions on $\left(x,y,z,t\right)$ which lead to $$\overline{g^{-1}\left(C\right)}= \bigcup_{(\alpha,\beta) \in C} \mathcal{D}_{\alpha,\beta}^{+} \cup \mathcal{D}_{\alpha,\beta}^{-}$$ 
where $$\mathcal{D}_{\alpha,\beta}^{+}= \left\lbrace (x,y,z,t) \in \dR^{3} \times \dR^{+} | \beta y=x^2 -\alpha \, \text{ and } \, \left(2-\alpha\right) z=\beta\right\rbrace $$ and $$\mathcal{D}_{\alpha,\beta}^{-}=\left\lbrace (x,y,z,t) \in \dR^{3}\times \dR^{-} | \beta y=-\alpha \, \text{ and } \, \left(2-\alpha\right) z=t^2+\beta\right\rbrace. $$
Thus, (\ref{ldp_q}) becomes
\begin{equation}\limsup_{T \rightarrow +\infty} \frac{1}{T} \log \dP\left(\overline{\theta}_T \in C \right) \leq - \inf_{(\alpha,\beta) \in C} \min \left\lbrace \inf_{\mathcal{D}_{\alpha,\beta}^{+}} \Lambda^{*} \, ;\inf_{\mathcal{D}_{\alpha,\beta}^{-}} \Lambda^{*}\right\rbrace.
\end{equation}
For all $\alpha \neq 2$ and $\beta \neq 0$, we easily rewrite $\mathcal{D}_{\alpha,\beta}^{+}= \left\lbrace \left(x,\frac{x^2-\alpha}{\beta},\frac{\beta}{2-\alpha}, t\right),x \in \dR, t \geq 0\right\rbrace $ and $\mathcal{D}_{\alpha,\beta}^{-}=\left\lbrace \left(x,-\frac{\alpha}{\beta}, \frac{t^2+\beta}{2-\alpha},t \right), x \in \dR, t\leq0  \right\rbrace$.\\
As the rate function $\Lambda^{*}$ given by Lemma~\ref{LDP_quad} is infinite for $t>0$, the infimum over $\mathcal{D}_{\alpha,\beta}^{+}$ reduces to the infimum over $\left\lbrace \left(x,\frac{x^2-\alpha}{\beta},\frac{\beta}{2-\alpha}, 0\right), x \in \dR \right\rbrace$ which is equal to $J_{a,b}\left(\alpha,\beta\right)$. Over $\mathcal{D}_{\alpha,\beta}^{-}$, we know by Lemma~\ref{LDP_quad} that $$\Lambda^{*}\left(x,-\frac{\alpha}{\beta},\frac{t^2+\beta}{2-\alpha},t\right)=\underset{\mathcal{D}_{d,f}}{\sup} \, \left\lbrace \frac{1}{4}\left(x\sqrt{d-b}-t\sqrt{\varphi(f)}\right)^2 + \psi_{d,f}\left(t\right)\right\rbrace $$ where $\psi_{d,f}(t)$ does not depend on $x$. Thus, $$\Lambda^{*}\left(x,-\frac{\alpha}{\beta},\frac{t^2+\beta}{2-\alpha},t\right) \geq \Lambda^{*}\left(0,-\frac{\alpha}{\beta},\frac{t^2+\beta}{2-\alpha},t\right) $$
and the infimum over $\mathcal{D}_{\alpha,\beta}^{-}$ is greater than the infimum over  $\left\lbrace \left(0,-\frac{\alpha}{\beta}, \frac{t^2+\beta}{2-\alpha},t \right), t\leq0  \right\rbrace$ which is equal to $K_{a,b}\left(\alpha,\beta\right)$.
\\
We now need to investigate the cases $\alpha=2$ and $\beta=0$ before concluding.
For $\beta=0$ and $\alpha \notin \left]0,2\right[$, we already know the value of $I_{a,b}$. If $\alpha \in \left]0,2\right[$, 
$$\mathcal{D}_{\alpha,0}^{+}= \left\lbrace (x,y,z,t) \in \dR^{3} \times \dR^{+} | x^2 =\alpha \, \text{ and } \, z=0 \right\rbrace \text{ and } \mathcal{D}_{\alpha,0}^{-}=\emptyset. $$
 And with the argument than before we obtain that for all $\alpha \in \left]0,2\right[$, $$I_{a,b}\left(\alpha,0\right)=J_{a,b}\left(\alpha,0\right)= \min \left\lbrace J_{a,b}\left(\alpha,0\right) ; K_{a,b}\left(\alpha,0\right) \right\rbrace$$ as $K_{a,b}\left(\alpha,0\right)$ is equal to infinity.
Now, for $\alpha=2$, we have already computed $I_{a,b}$ for all $\beta \geq 0$. For $\beta <0$,
$$\mathcal{D}_{2,\beta}^{+}=\emptyset \text{ and }\mathcal{D}_{2,\beta}^{-}= \left\lbrace (x,y,z,t) \in \dR^{3} \times \dR^{-} | \beta y=-2 \, \text{ and } \, t^2=-\beta \right\rbrace . $$
 We obtain that
$$\mathcal{D}_{2,\beta}^{-}=\left\lbrace \left(x,-\frac{2}{\beta},z,t \right), x \in \dR, z\in \dR, t\leq0  \right\rbrace$$
and with the same argument than before, the infimum over this subset is greater than the infimum over $\left\lbrace \left(0,-\frac{2}{\beta},z,t \right), z \in \dR, t\leq0  \right\rbrace$, which is equal to $K_{a,b}\left(2,\beta\right)$. Thus, as $J_{a,b}\left(2,\beta\right)$ is equal to infinity, we can conclude that
$$I_{a,b}\left(2,\beta\right)= \min \left\lbrace J_{a,b}\left(2,\beta\right) ; K_{a,b}\left(2,\beta\right) \right\rbrace.$$ 
\end{proof}

\section*{Appendix A: Proofs of the CLT for the two couples of simplified estimators}
\renewcommand{\thesection}{\Alph{section}}
\renewcommand{\theequation}{\thesection.\arabic{equation}}
\setcounter{section}{1}
\setcounter{equation}{0}

The key to obtain those results is Slutsky's lemma. Indeed, we have
\begin{equation}\label{norm1}\sqrt{T}\begin{pmatrix}\wt{a}_T - a \\
\wt{b}_T -b \end{pmatrix} =\sqrt{T}\begin{pmatrix}\wh{a}_T - a \\
\wh{b}_T -b \end{pmatrix} + \sqrt{T}\begin{pmatrix} - \frac{S_T L_T}{V_T} \\
\frac{L_T}{V_T} \end{pmatrix} 
\end{equation}
where
$$\sqrt{T}\begin{pmatrix}\wh{a}_T - a \\
\wh{b}_T -b \end{pmatrix} \xrightarrow{\mathcal{L}} \mathcal{N}(0,4C^{-1}) \text{ with } C= \begin{pmatrix}
 \frac{-b}{a-2} & -1 \\
 -1 & -\frac{a}{b}
 \end{pmatrix},$$ and 
we show that the right-hand side of (\ref{norm1}) converges to zero in probability.
 Namely, it is well known (see for instance Lemma 3 of \cite{Ov}) that $S_T$ converges almost surely to $-\frac{a}{b}$ and $V_T$ to $\frac{2}{a-2}$. And for all $\varepsilon >0$, we have
$$ \begin{array}{lcl}
 \dP\left(|\frac{\log X_T}{\sqrt{T}}| \geq \varepsilon \right) &=& \dP\left(|\log X_T| \geq \sqrt{T}\varepsilon \right)\\
 &\leq & \dP\left(\log X_T \geq \sqrt{T}\varepsilon \right) + \dP\left(-\log X_T \geq \sqrt{T}\varepsilon \right)\\
 & \leq & \dE\left(X_T\right) e^{-\sqrt{T}\varepsilon}+ \dE\left(X_T^{-1}\right) e^{-\sqrt{T}\varepsilon}\\
 &\xrightarrow{T\to +\infty}& 0
 \end{array}$$

because $\dE\left(X_T\right)$ converges almost surely to $\dE\left(X_{\infty}\right)=-\frac{a}{b}$ (see \cite{Ov} Lemma 3) and, as the parameter $a$ is supposed greater than $2$, we obtain from Proposition 3 in \cite{KAB2} that $$\dE\left(X_T^{-1}\right) \rightarrow -\frac{b}{2}\frac{\Gamma(a/2-1)}{\Gamma(a/2)} = \dE\left(X_{\infty}^{-1}\right).$$
This gives the announced convergence in probability to zero. Thus, with Slutsky's lemma, the simplified estimators $(\wt{a}_T,\wt{b}_T)$ satisfy the same asymptotic normality result than the MLE.\\
 Similarly, for the second couple of simplified estimators, we have \begin{equation}\label{norm2}\sqrt{T}\begin{pmatrix}\widecheck{a}_T - a \\
\widecheck{b}_T -b \end{pmatrix} =\sqrt{T}\begin{pmatrix}\wh{a}_T - a \\
\wh{b}_T -b \end{pmatrix} + \sqrt{T}\begin{pmatrix} \frac{X_T}{T V_T} \\
-\frac{X_T \Sigma_T}{T V_T} \end{pmatrix}. 
\end{equation}
As $\Sigma_T$ converges almost surely to $-\frac{b}{a-2}$ 
we only have to show that $X_T/\sqrt{T}$ converges to zero in probability. For all $\epsilon >0$,
 $$\dP\left(\left| \frac{X_T}{\sqrt{T}}\right| \geq \varepsilon \right) =\dP\left( X_T \geq \sqrt{T}\varepsilon \right) \leq  \frac{\dE\left(X_T\right)}{\sqrt{T}\varepsilon} \xrightarrow{T\to +\infty} 0
 $$
 with the same argument than before.
Thus $(\widecheck{a}_T,\widecheck{b}_T)$ also satisfies the same asymptotic normality result.\demend


\section*{Appendix B: proof of Lemma ~\ref{LDPC}}

\renewcommand{\thesection}{\Alph{section}}
\renewcommand{\theequation}{\thesection.\arabic{equation}}
\setcounter{section}{2}
\setcounter{equation}{0}

  We apply the G\"artner-Ellis theorem (see \cite{DeZ}). It is easy to deduce from Theorem ~\ref{CGFquadruplet} that the pointwise limit $\wt{\Lambda}$ of the normalized cumulant generating function $\wt{\Lambda}_T$ of the couple $(S_T,\Sigma_T)$ is given by
\begin{equation}\wt{\Lambda}(\mu,\nu)= -\frac{d}{2}\left(1+f\right)-\frac{ab}{4}
\end{equation}
where $d=\sqrt{b^2-8\mu}$ and $f= \sqrt{\left(\frac{a}{2}-1\right)^2 -2\nu}$. We easily get from Lemma~\ref{steep} that the function $\wt{\Lambda}$ is steep.
To obtain the rate function $I$, we just have to compute the Fenchel-Legendre transform of $\wt{\Lambda}$: \begin{equation}\label{defI}I(x,y) = \underset{\mu<\frac{b^2}{8}, \, \nu < \frac{(a-2)^2}{8}}{\sup} \left\lbrace x\mu + y \nu - \wt{\Lambda}(\mu,\nu) \right\rbrace.
\end{equation}
First we note that 
\begin{equation}\label{r1} \text{ for } x \leq 0 \text{ or } y \leq 0, \: \:  I(x,y)= +\infty,
\end{equation}
 because $\wt{\Lambda}(\mu,\nu)$ tends to $-\infty$ as $\nu$ or $\mu$ tends to $-\infty$. Only the case $x>0$ and $y>0$ remains to be studied. \\
We look for the critical points.
If $xy-1 \neq 0$, we obtain 
\begin{equation}
\begin{pmatrix}
d_0\\
f_0
\end{pmatrix}= \frac{1}{xy-1} \begin{pmatrix}
2y \\
1
\end{pmatrix}. 
\end{equation}
As $d_0$ and $f_0$ must both be positive, the solution is in the domain if and only if $xy-1>0$. It is easy to check that this critical point corresponds to a maximum of $\wt{\Lambda}$. 
 Using the fact that $\mu=\frac{b^2-d^2}{8}$ et $\nu = \frac{(\frac{a}{2}-1)^2-f^2}{2}$ and replacing it into (\ref{defI}), we get 
 \begin{equation}\label{r2}
 I(x,y)= \frac{y}{2(xy-1)}+\frac{b^2}{8}x+\frac{(a-2)^2}{8}y + \frac{ab}{4} \, .
 \end{equation}
 To conclude, we need to examine the case $x>0$, $y>0$ and $xy-1<0$. 
 We already know that for $\nu$ and $\mu$ tending to $-\infty$, $-\wt{\Lambda}(\mu,\nu)$ tends to $+\infty$. However, as $x$ and $y$ are non negative, we cannot conclude directly. But it is possible to find a direction in which $\wt{\Lambda}$ dominate the expression. Note that, for $-\nu$ and $-\mu$ large enough,
 $$x\mu+y\nu-\wt{\Lambda}(\mu,\nu) \sim x\mu+y\nu + \sqrt{\mu \nu}.$$
Let $k>0$ and $\nu=k \mu$. 
We just have to find a $k>0$ that satisfies 
 If $y<1$ then all $k>-\frac{x}{y-1}$ fit. Else, if $y \geq 1$, necessarily $x<1$ and we use the same argument with $\mu = k \nu$ this time. We have found directions for which $x\mu+y\nu -\wt{\Lambda}(\mu,\nu)$ tends to $+\infty$, so that the supremum itself is equal to $+\infty$. And, 
 \begin{equation} \label{r3} \text{for } x>0, y>0 \text{ such that } xy-1<0, 
 \: \: \: I(x,y)=+\infty.
 \end{equation} 
 Combining \eqref{r1}, \eqref{r2} and \eqref{r3}, we obtain the announced result.\demend
  
\section*{Appendix C: Proofs of Lemmas ~\ref{J1gammaNeg} to ~\ref{J2lambdaPos}} 
\renewcommand{\thesection}{\Alph{section}}
\renewcommand{\theequation}{\thesection.\arabic{equation}}
\setcounter{section}{3}
\setcounter{equation}{0}

The four following proofs rely on lower and upper bounds for the modified Bessel function of the first kind, given by formula (6.25) of \cite{Luke}.
More precisely, for all $z>0$ and $\nu>-\frac{1}{2}$, we have:
\begin{equation}\label{InegI}
1<\left(\frac{2}{z}\right)^{\nu} \Gamma\left(\nu+1\right) I_{\nu}\left(z\right)<e^z.
\end{equation}

\subsection*{C.1. Proof of Lemma ~\ref{J1gammaNeg}}
 It is easy to deduce the following upper and lower bounds from ~(\ref{InegI}):
\begin{equation}\label{InegII}
\frac{(\beta_T \sqrt{y})^f}{2^f \, \Gamma(f+1)} \leq I_{f}\left(\beta_T \sqrt{y} \right) \leq \frac{(\beta_T \sqrt{y})^f}{2^f \Gamma(f+1)} e^{\beta_T \sqrt{y}}.
\end{equation}
Replacing it into the expression of $H_T$ given by (\ref{defHT}) leads to:
 \begin{equation}\frac{2^f \Gamma(1+f)}{\beta_T^f} \, H_T \geq \int_0^1{e^{\lambda \sqrt{Ty}-\gamma \sqrt{-T \log y}-\alpha_T y}\,  y^{\frac{2f+a-2}{4}} \,   \mathrm{d}y}
 \end{equation}
 and \begin{equation}
 \frac{2^f \Gamma(1+f)}{\beta_T^f} \, H_T \leq   \int_0^1{e^{\beta_T \sqrt{y} + \lambda \sqrt{Ty}-\gamma \sqrt{-T \log y}-\alpha_T y}\,  y^{\frac{2f+a-2}{4}} \,   \mathrm{d}y}. 
 \end{equation}
 We consider separately the integrals over $[0,\frac{1}{T}]$ and over $[\frac{1}{T},1]$. On the one hand,  \\
$\begin{array}{lcl}
\displaystyle \int_{\frac{1}{T}}^{1}{e^{\beta_T \sqrt{y} + \lambda \sqrt{Ty}-\gamma \sqrt{-T \log y}-\alpha_T y}\,  y^{\frac{2f+a-2}{4}} \,   \mathrm{d}y}
& \leq &\displaystyle \int_{\frac{1}{T}}^{1}{e^{\beta_T \sqrt{y} + \lambda \sqrt{Ty}-\gamma \sqrt{-T \log y}} \,  \mathrm{d}y} \\
& \leq & 1 \times \displaystyle \sup_{[\frac{1}{T},1]}\left( e^{\beta_T \sqrt{y} + \lambda \sqrt{Ty}-\gamma \sqrt{-T \log y}}\right) \\
& \leq & e^{\lambda \sqrt{T}+\beta_T-\gamma \sqrt{-T \log \frac{1}{T}}} \\
& = & O\left(e^{(\lambda-\gamma) \sqrt{T}\sqrt{\log T}}\right).\\
\end{array}$\\  
 On the other hand, with the same argument, we show that $$\int_{\frac{1}{T}}^1{e^{\lambda \sqrt{Ty}-\gamma \sqrt{-T \log y}-\alpha_T y}\,  y^{\frac{2f+a-2}{4}} \,   \mathrm{d}y} = O\left(e^{(\lambda-\gamma) \sqrt{T}\sqrt{\log T}}\right).$$
Over $[0,\frac{1}{T}]$, $e^{\lambda \sqrt{Ty}}$, $e^{\beta_T \sqrt{y}}$ and $e^{-\alpha_T y}$ are bounded. So, as $\alpha_T$ and $\beta_T$ are always positive, we have the following bounds
$$ \int_0^{\frac{1}{T}}{e^{\lambda \sqrt{Ty}-\gamma \sqrt{-T \log y}-\alpha_T y}\,  y^{\frac{2f+a-2}{4}} \,   \mathrm{d}y} \geq e^{-\alpha_T/T -|\lambda|} \int_0^{\frac{1}{T}}{e^{-\gamma \sqrt{-T \log y}}\,  y^{\frac{2f+a-2}{4}} \,   \mathrm{d}y},$$
$$\int_0^{\frac{1}{T}}{e^{\beta_T \sqrt{y} + \lambda \sqrt{Ty}-\gamma \sqrt{-T \log y}-\alpha_T y}\,  y^{\frac{2f+a-2}{4}} \,   \mathrm{d}y} \leq e^{|\lambda|+\alpha_T/T +\beta_T/\sqrt{T}} \int_0^{\frac{1}{T}}{e^{-\gamma \sqrt{-T \log y}}\,  y^{\frac{2f+a-2}{4}} \,   \mathrm{d}y}. $$
Using the change of variable given by $z= \sqrt{-\log y} + \frac{\gamma \sqrt{T}}{2g}$, where $g=\frac{2f+a+2}{4}$, we obtain:
$$\begin{array}{lcl}
\displaystyle \int_0^{\frac{1}{T}}{e^{-\gamma \sqrt{-T \log y}}\,  y^{g-1} \,   \mathrm{d}y}&= &2 e^{\frac{\gamma ^2 T}{4g}} \displaystyle \int_{\sqrt{\log T} + \frac{\gamma \sqrt{T}}{2g}}^{+\infty}{e^{-g \, z^2}\left(z-\frac{\gamma \sqrt{T}}{2g}\right) \, \mathrm{d}z}\\
&=& 2 e^{\frac{\gamma ^2 T}{4g}} \left(\displaystyle \int_{\sqrt{\log T} + \frac{\gamma \sqrt{T}}{2g}}^{+\infty}{e^{-g\, z^2}\, z \, \mathrm{d}z} - \frac{\gamma \sqrt{T}}{2g} \, \displaystyle \int_{\sqrt{\log T} + \frac{\gamma \sqrt{T}}{2g}}^{+\infty}{e^{-g\, z^2} \, \mathrm{d}z}\right). \\
\end{array}$$
By dominated convergence as $\gamma <0$, the first integral tends to zero when T goes to infinity and the second one tends to the positive constant $\sqrt{\frac{\pi}{g}}$. This leads to the following bounds, for $T$ large enough:
$$\int_0^{\frac{1}{T}}{e^{-\gamma \sqrt{-T \log y}}\,  y^{g-1} \,   \mathrm{d}y} \leq 2 \sqrt{\frac{\pi}{g}} \, \frac{|\gamma| \sqrt{T}}{g} \, \exp\left(\frac{\gamma^2 T}{4g}\right)$$
$$\int_0^{\frac{1}{T}}{e^{-\gamma \sqrt{-T \log y}}\,  y^{g-1} \,   \mathrm{d}y} \geq \frac{1}{2} \sqrt{\frac{\pi}{g}} \, \frac{|\gamma| \sqrt{T}}{g} \, \exp\left(\frac{\gamma^2 T}{4g}\right).$$ Combined with the result over $[\frac{1}{T},1]$, it gives the announced result.

\subsection*{C.2. Proof of Lemma ~\ref{J1gammaPos}}
The upper bound easily follows from ~(\ref{InegII}). Actually, for all $\gamma>0$ and $\lambda \in \mathbb{R}$, we obtain
\begin{equation}
\begin{array}{lcl}
\displaystyle  H_T &\leq & \frac{\beta_T^f}{\Gamma(f+1) 2^f} \displaystyle \int_0^1{e^{\left(\lambda \sqrt{T}+\beta_T\right) \sqrt{y}-\gamma \sqrt{-T\log y}} \, y^{\frac{2f+a-2}{4}}\, \mathrm{d}y}\\
&\leq & \frac{\beta_T^f}{\Gamma(f+1) 2^f} \, \displaystyle e^{|\lambda|\sqrt{T}+\beta_T}.\\
\end{array}
\end{equation}
\\
Besides, using the lower bound of ~(\ref{InegII}), we clearly have, for all $\gamma >0$ and $\lambda \in \mathbb{R}$,
\begin{equation} \displaystyle H_T\geq \frac{\beta_T^f \, e^{-\alpha_T}}{\Gamma(f+1) 2^f} \,  \displaystyle \int_0^1{e^{\lambda \sqrt{T} \sqrt{y}-\gamma \sqrt{-T\log y}} \, y^{\frac{2f+a-2}{4}}\, \mathrm{d}y}.
\end{equation}
To obtain the announced lower bound, we need to consider separately the integral over $[0,\frac{1}{T}]$ and over $[\frac{1}{T},1]$. On the one hand, the integral over $[\frac{1}{T},1]$ is easy to handle. 
\begin{equation}\label{LB1}\int_{\frac{1}{T}}^{1}{e^{\lambda \sqrt{Ty}-\gamma \sqrt{-T\log y}} \, y^{\frac{2f+a-2}{4}}\, \mathrm{d}y} \geq \varepsilon \, e^{-\gamma \sqrt{T}\sqrt{\log T}} \,  T^{-\frac{2f+a-2}{4}},
\end{equation}
where $\varepsilon=\exp\left(\lambda \mathbf{1}_{\lambda \geq 0}+\lambda \sqrt{T} \mathbf{1}_{\lambda < 0}\right)$.
On the other hand, 
$$\int_0^{\frac{1}{T}}{e^{\lambda \sqrt{T}\sqrt{y}-\gamma \sqrt{-T \log y}}\,  y^{\frac{2f+a-2}{4}} \,   \mathrm{d}y} \geq e^{-|\lambda|} \, \int_0^{\frac{1}{T}}{e^{-\gamma \sqrt{-T \log y}}\,  y^{\frac{2f+a-2}{4}} \,   \mathrm{d}y}$$
Using the variable change $z= \sqrt{-\log y} + \frac{\gamma \sqrt{T}}{2g}$, where we recall that $g= \frac{2f+a+2}{4}$, we obtain
$$\displaystyle \int_0^{\frac{1}{T}}{e^{-\gamma \sqrt{-T \log y}}\,  y^{\frac{2f+a-2}{4}} \,   \mathrm{d}y}= 2 \, e^{\frac{\gamma ^2 T}{4g}} \displaystyle \int_{\sqrt{\log T} + \frac{\gamma \sqrt{T}}{2g}}^{+\infty}{e^{-g \, z^2}\left(z-\frac{\gamma \sqrt{T}}{2g}\right) \, \mathrm{d}z}.$$
Firstly, we clearly have
$$\displaystyle \int_{\sqrt{\log T} + \frac{\gamma \sqrt{T}}{2g}}^{+\infty}{e^{-g \, z^2} \, z \, \mathrm{d}z}= \frac{1}{2g} \exp \left(-g \, \left(\sqrt{\log T} + \frac{\gamma \sqrt{T}}{2g} \right)^2\right).$$
Besides, for the second part of the integral, we have
$$\begin{array}{lcl}
 \displaystyle \int_{\sqrt{\log T} + \frac{\gamma \sqrt{T}}{2g}}^{+\infty}{e^{-g \, z^2}\, \mathrm{d}z}&=& \displaystyle \int_{\sqrt{\log T} + \frac{\gamma \sqrt{T}}{2g}}^{+\infty}{e^{-g \, z^2}\, z \times \frac{1}{z} \mathrm{d}z} \\
 & \leq & \frac{1}{\sqrt{\log T} + \frac{\gamma \sqrt{T}}{2g}} \displaystyle \int_{\sqrt{\log T} + \frac{\gamma \sqrt{T}}{2g}}^{+\infty}{e^{-g\, z^2}\, z \mathrm{d}z} \\
 &=&  \frac{1}{2g\sqrt{\log T} + \gamma \sqrt{T}} \exp \left(-g \, \left(\sqrt{\log T} + \frac{\gamma \sqrt{T}}{2g} \right)^2\right).
\end{array} $$
Thus, for any positive $\gamma$, we have the following lower bound:
$$\displaystyle \int_0^{\frac{1}{T}}{e^{-\gamma \sqrt{-T \log y}}\,  y^{\frac{2f+a-2}{4}} \,   \mathrm{d}y} \geq \frac{1}{g} \left(1-\frac{\gamma \sqrt{T}}{2g\sqrt{\log T} +\gamma \sqrt{T}}\right)   \, \exp \left( -g \log T - \gamma \sqrt{T} \sqrt{\log T}\right). $$
Combined with (\ref{LB1}), 
 this leads to
$$H_T \geq \frac{\left(\beta_T\right)^f \, e^{-\alpha_T}}{2^f \, \Gamma\left(1+f\right)} \left(\varepsilon+\frac{e^{-|\lambda|}}{g} \left(1-\frac{\gamma \sqrt{T}}{2g\sqrt{\log T} +\gamma \sqrt{T}}\right)\right) \, \exp \left(-\gamma \sqrt{T} \sqrt{\log T}- g \, \log T \right). $$

\subsection*{C.3. Proof of Lemma ~\ref{J2lambdaNeg}}
Using the inequality ~(\ref{InegII}) for the modified Bessel function $I_f$ in (\ref{defKT}), we obtain 
\begin{equation}K_T\leq \frac{\left(\beta_T\right)^f}{2^f \Gamma(1+f)} \int_1^{+\infty}{e^{\beta_T \sqrt{y} + \lambda \sqrt{Ty}-\alpha_T y}\,  y^{\gamma+\frac{2f+a-2}{4}} \, \mathrm{d}y} \end{equation}
and
\begin{equation}K_T\geq \frac{\left(\beta_T\right)^f}{2^f \Gamma(1+f)} \int_1^{+\infty}{e^{ \lambda \sqrt{Ty}-\alpha_T y}\,  y^{\gamma+\frac{2f+a-2}{4}} \, \mathrm{d}y}. 
\end{equation}
To go further, we need to consider the sign of the exponent $\gamma+\frac{2f+a-2}{4}$.\\
$\bullet$ If $\gamma+\frac{2f+a-2}{4} \leq 0$: using the fact that $y^{\gamma+\frac{2f+a-2}{4}} \leq 1$ and with the change of variable $u=\sqrt{y}-\frac{\lambda \sqrt{T}+\beta_T}{2\alpha_T}$, we obtain the following asymptotic behaviour for the upper bound of $K_T$,
$$\begin{array}{lcl}
\displaystyle K_T& \leq & \displaystyle \frac{\left(\beta_T\right)^f}{2^f \Gamma(1+f)} \int_1^{+\infty}{e^{\left(\beta_T +\lambda \sqrt{T}\right) \sqrt{y}-\alpha_T y}\, \mathrm{d}y}\\
&=& \frac{2 \left(\beta_T\right)^f}{2^f \Gamma(1+f)} e^{\frac{\left(\lambda \sqrt{T}+\beta_T\right)^2}{4\alpha_T}}\displaystyle \int_{1-\frac{\lambda \sqrt{T}+\beta_T}{2\alpha_T}}^{+\infty}{e^{-\alpha_T u^2}\left(u+\frac{\lambda \sqrt{T}+\beta_T}{2\alpha_T}\right) \, \mathrm{d}u} \\
&\leq & \frac{2 \left(\beta_T\right)^f}{2^f \Gamma(1+f)}  e^{\frac{\left(\lambda \sqrt{T}+\beta_T\right)^2}{4\alpha_T}} \left(A_1+\frac{\beta_T}{2\alpha_T} A_2 \right)
\end{array}$$
where $A_1$ and $A_2$ are given by: \begin{equation}A_1= \int_{1-\frac{\lambda \sqrt{T}+\beta_T}{2\alpha_T}}^{+\infty}{e^{-\alpha_T u^2}\, u \, \mathrm{d}u} = \frac{1}{2\alpha_T}e^{-\alpha_T \left( 1-\frac{\lambda \sqrt{T}+\beta_T}{2\alpha_T} \right)^2},
\end{equation}
\begin{equation}A_2=\int_{1-\frac{\lambda \sqrt{T}+\beta_T}{2\alpha_T}}^{+\infty}{e^{-\alpha_T u^2}\, \mathrm{d}u} = \int_{1-\frac{\lambda \sqrt{T}+\beta_T}{2\alpha_T}}^{+\infty}{e^{-\alpha_T u^2}\, u \times \frac{1}{u} \, \mathrm{d}u}
\leq \displaystyle  \frac{1}{1-\frac{\lambda \sqrt{T}+\beta_T}{2\alpha_T}} A_1.
\end{equation}
Thus:\begin{equation}\begin{array}{lcl}
\displaystyle K_T &\leq & \frac{2 \left(\beta_T\right)^f}{2^f \Gamma(1+f)}  e^{\frac{\left(\lambda \sqrt{T}+\beta_T\right)^2}{4\alpha_T}} \frac{1}{2\alpha_T} e^{-\alpha_T \left( 1-\frac{\lambda \sqrt{T}+\beta_T}{2\alpha_T} \right)^2}\left(1+\displaystyle \frac{\frac{\beta_T}{2\alpha_T}}{1-\frac{\lambda \sqrt{T}+\beta_T}{2\alpha_T}}\right)\\
&=& \frac{\left(\beta_T\right)^f}{2^f \Gamma(1+f) \alpha_T } \left(1+\displaystyle \frac{\frac{\beta_T}{2\alpha_T}}{1-\frac{\lambda \sqrt{T}+\beta_T}{2\alpha_T}}\right) \exp \left(\lambda \sqrt{T}+\beta_T-\alpha_T \right).
\end{array}
\end{equation}
$\bullet$ If $\gamma+\frac{2f+a-2}{4} > 0$: with formula 3.462(1) in \cite{GR}, we get
$$\begin{array}{lcl}
K_T &\leq& \frac{\left(\beta_T\right)^f}{2^f \Gamma(1+f)}  \left(2\left(2\alpha_T\right)^{-(\gamma+g)} \Gamma\left(2\gamma+2g\right) e^{\frac{(\lambda\sqrt{T}+\beta_T)^2}{8\alpha_T}} D_{-2(\gamma+g)} \left(\frac{-\beta_T-\lambda\sqrt{T}}{\sqrt{2\alpha_T}}\right)\right.\\
& & \left. \, \, \, \, \, \, \, \, \, \, \, \, \, \, \, \, \, \, \, \, \, \, \, \, \, \, \, \, \, \, \displaystyle - \int_0^1{e^{\beta_T \sqrt{y} + \lambda \sqrt{Ty}-\alpha_T y}\,  y^{\gamma+g-1} \, \mathrm{d}y}\right)\\
& \leq & \frac{2\left(\beta_T\right)^f \left(2\alpha_T\right)^{-(\gamma+g)} \Gamma\left(2\gamma+2g\right)}{2^f \Gamma(1+f)}   e^{\frac{(\lambda\sqrt{T}+\beta_T)^2}{8\alpha_T}} D_{-2(\gamma+g)} \left(\frac{-\beta_T-\lambda\sqrt{T}}{\sqrt{2\alpha_T}}\right),
\end{array}$$
where $D_{-2(\gamma+g)}$ is the parabolic cylinder function defined by 9.250 in \cite{GR}.
But, if $\lambda <0$, for $T$ large enough, $-\beta_T-\lambda\sqrt{T} >0$ so formula 9.246(1) of \cite{GR} gives
 $$D_{- 2(\gamma+g)}\left(\frac{-\beta_T-\lambda\sqrt{T}}{\sqrt{2\alpha_T}}\right) \sim \left(\frac{-\beta_T-\lambda \sqrt{T}}{\sqrt{2\alpha_T}} \right)^{-2(\gamma+g)}e^{-\frac{(\lambda\sqrt{T}+\beta_T)^2}{8\alpha_T}}.$$ Thus $K_T = O(\beta_T^f)$ as $T$ goes to infinity.
If $\lambda=0$, we use formula 9.246(2) in \cite{GR}. $$D_{-2(\gamma+g)} \left(\frac{-\beta_T}{\sqrt{2\alpha_T}}\right) \sim \frac{\sqrt{2\pi}}{\Gamma\left(2\gamma+2g\right)} \left(\frac{\beta_T}{\sqrt{2\alpha_T}} \right)^{2\gamma+2g-1} e^{\frac{(\beta_T)^2}{8\alpha_T}}.$$
This leads to the same conclusion: as $T$ goes to infinity,
\begin{equation}
K_T = O\left(\left(\beta_T\right)^f\right).
\end{equation}
Otherwise, for the proof of Lemma ~\ref{CGFquadruplet}, we also need a lower bound for $K_T$ when $\gamma \geq 0$. We note that over $[1,+\infty[$, $y^{\gamma +g-1}\geq 1$, so that
\begin{equation}K_T \geq \frac{2^{-f} \left(\beta_T\right)^f}{\Gamma(1+f)} \int_1^{+\infty}{e^{\lambda \sqrt{T}\sqrt{y}-\alpha_T y}\, \mathrm{d}y }.
\end{equation}
With the change of variable given by $u=\sqrt{y}-\frac{\lambda \sqrt{T}}{2\alpha_T}$, it becomes
\begin{equation}\label{LB4}\begin{array}{lcl}
K_T &\geq& \frac{2^{1-f} \left(\beta_T\right)^f}{\Gamma(1+f)} e^{\frac{\lambda^2 T}{4\alpha_T}} \displaystyle \int_{1-\frac{\lambda \sqrt{T}}{2\alpha_T}}^{+\infty}{e^{-\alpha_T u^2}(u+\frac{\lambda \sqrt{T}}{2\alpha_T})\, \mathrm{d}u }\\
&\geq& \frac{2^{1-f} \left(\beta_T\right)^f}{\Gamma(1+f)} e^{\frac{\lambda^2 T}{4\alpha_T}}\left(\displaystyle \int_{1-\frac{\lambda \sqrt{T}}{2\alpha_T}}^{+\infty}{e^{-\alpha_T u^2} u\, \mathrm{d}u } +\frac{\lambda \sqrt{T}}{2\alpha_T}\int_{1-\frac{\lambda \sqrt{T}}{2\alpha_T}}^{+\infty}{e^{-\alpha_T u^2}\, \mathrm{d}u }\right)\\
\end{array}
\end{equation}
However, the first integral is easily computable:
\begin{equation}\label{LB2}\int_{1-\frac{\lambda \sqrt{T}}{2\alpha_T}}^{+\infty}{e^{-\alpha_T u^2} u\, \mathrm{d}u}= \frac{1}{2\alpha_T} \exp\left(-\alpha_T \left(1-\frac{\lambda\sqrt{T}}{2\alpha_T}\right)^2\right)
\end{equation}
and for the second one, we clearly have the following upper bound,
\begin{equation}\label{LB3}\begin{array}{lcl}
\int_{1-\frac{\lambda \sqrt{T}}{2\alpha_T}}^{+\infty}{e^{-\alpha_T u^2} \, \mathrm{d}u}&=& \int_{1-\frac{\lambda \sqrt{T}}{2\alpha_T}}^{+\infty}{e^{-\alpha_T u^2} u \times \frac{1}{u}\, \mathrm{d}u} \\
& \leq & \frac{1}{1-\frac{\lambda \sqrt{T}}{2\alpha_T}} \int_{1-\frac{\lambda \sqrt{T}}{2\alpha_T}}^{+\infty}{e^{-\alpha_T u^2} u\, \mathrm{d}u}\\
& \leq & \frac{1}{1-\frac{\lambda \sqrt{T}}{2\alpha_T}} \frac{1}{2\alpha_T} \exp\left(-\alpha_T \left(1-\frac{\lambda\sqrt{T}}{2\alpha_T}\right)^2\right).
\end{array}
\end{equation}
Using the fact that $\lambda<0$ and combining (\ref{LB4}), (\ref{LB2}) and (\ref{LB3}), we show the announced result:
\begin{equation}K_T \geq \frac{2^{1-f} \left(\beta_T\right)^f}{\Gamma(1+f)} \frac{e^{-\alpha_T}}{2\alpha_T} \left(1+\frac{\frac{\lambda \sqrt{T}}{2\alpha_T}}{1-\frac{\lambda \sqrt{T}}{2\alpha_T}}\right) \exp \left(\lambda \sqrt{T}\right).
\end{equation}
All those results are still true for $\lambda=0$. 

\subsection*{C.4. Proof of Lemma ~\ref{J2lambdaPos}}
As previously, we use ~(\ref{InegII}) to find lower and upper bounds for $K_T$ as $T$ goes to infinity, for all $\lambda >0$. We consider two cases depending on the sign of the exponent $\gamma +g-1$.\\
$\bullet$ If $\gamma+g-1 \leq 0$:
As in the proof of Lemma ~\ref{J2lambdaNeg}, we have the following upper bound \begin{equation}K_T \leq \frac{2 \left(\beta_T\right)^f}{2^f \Gamma(1+f)}  e^{\frac{\left(\lambda \sqrt{T}+\beta_T\right)^2}{4\alpha_T}} \left(A_1+\frac{\beta_T}{2\alpha_T} A_2 \right)
\end{equation}
where $A_1= \displaystyle \int_{1-\frac{\lambda \sqrt{T}+\beta_T}{2\alpha_T}}^{+\infty}{e^{-\alpha_T u^2}\, u \, \mathrm{d}u}$ tends to zero for all $\lambda >0$ by dominated convergence, and $A_2=\displaystyle \int_{1-\frac{\lambda \sqrt{T}+\beta_T}{2\alpha_T}}^{+\infty}{e^{-\alpha_T u^2}\, \mathrm{d}u}$ tends to the positive constant $2\sqrt{\frac{\pi}{d-b}}$. Thus, for $T$ large enough, 
 \begin{equation}K_T \leq 2^{1-f} \sqrt{\frac{\pi}{d-b}}\,  \frac{\left(\beta_T\right)^{1+f}}{\alpha_T \, \Gamma(1+f)} \, e^{\frac{\left(\lambda \sqrt{T}+\beta_T\right)^2}{4\alpha_T}}.
 \end{equation}
 For the lower bound, with the change of variable given by $u=\sqrt{y}-\frac{\lambda \sqrt{T}}{2\alpha_T}$, we obtain
 $$\begin{array}{lcl}
 K_T &\geq& \frac{2^{-f} \left(\beta_T\right)^f}{\Gamma(1+f)} \displaystyle  \int_1^{+\infty}{e^{\lambda \sqrt{T}\sqrt{y}-\alpha_T y}\,  y^{\gamma+g-1} \, \mathrm{d}y } \\
 & = & \frac{2^{1-f} \left(\beta_T\right)^f}{\Gamma(1+f)} e^{\frac{\lambda^2 T}{4\alpha_T}} \displaystyle \int_{1-\frac{\lambda \sqrt{T}}{2\alpha_T}}^{+\infty}{e^{-\alpha_T u^2} \left( u+\frac{\lambda \sqrt{T}}{2\alpha_T}\right)^{2(\gamma+g-1)} \left( u+\frac{\lambda \sqrt{T}}{2\alpha_T}\right)\, \mathrm{d}u}\\
 &\geq& \frac{2^{1-f} \left(\beta_T\right)^f}{\Gamma(1+f)} e^{\frac{\lambda^2 T}{4\alpha_T}} \displaystyle  \int_{1-\frac{\lambda \sqrt{T}}{2\alpha_T}}^{+\infty}{e^{-\alpha_T u^2} \left( 2 \max \left\lbrace u;\frac{\lambda \sqrt{T}}{2\alpha_T}\right\rbrace\right)^{2(\gamma+g-1)}\, \mathrm{d}u}\\
 &\geq& \frac{2^{1-f} \left(\beta_T\right)^f}{\Gamma(1+f)} e^{\frac{\lambda^2 T}{4\alpha_T}} \left(\left(\frac{2\lambda \sqrt{T}}{2\alpha_T}\right)^{2(\gamma+g-1)} \displaystyle \int_{1-\frac{\lambda \sqrt{T}}{2\alpha_T}}^{\frac{\lambda \sqrt{T}}{2\alpha_T}}{e^{-\alpha_T u^2} \, \mathrm{d}u}+\int_{\frac{\lambda \sqrt{T}}{2\alpha_T}}^{+\infty}{e^{-\alpha_T u^2} \left(2u\right)^{2(\gamma+g-1)}\, \mathrm{d}u} \right)\\
 &\geq& \frac{2^{1-f} \left(\beta_T\right)^f}{\Gamma(1+f)} e^{\frac{\lambda^2 T}{4\alpha_T}} \left(\frac{\lambda \sqrt{T}}{\alpha_T}\right)^{2(\gamma+g-1)}\displaystyle \int_{1-\frac{\lambda \sqrt{T}}{2\alpha_T}}^{\frac{\lambda \sqrt{T}}{2\alpha_T}}{e^{-\alpha_T u^2} \, \mathrm{d}u}.
 \end{array}$$
 By dominated convergence, the last integral converges as $T$ tends to infinity to the positive constant $2\sqrt{\frac{\pi}{d-b}}$. For $T$ large enough, this leads to \begin{equation}
 K_T \geq 2^{1-f} \sqrt{\frac{\pi}{d-b}} \frac{\left(\beta_T\right)^f}{\Gamma(1+f)} \left(\frac{\lambda \sqrt{T}}{\alpha_T}\right)^{2(\gamma+g-1)} e^{\frac{\lambda^2 T}{4\alpha_T}}.
 \end{equation}
 
 $\bullet$ If $\gamma +g-1 > 0$: over $[1,+\infty[$, we notice that $y^{\gamma +g-1}\geq 1$. Thus
\begin{equation}K_T \geq \frac{2^{-f} \left(\beta_T\right)^f}{\Gamma(1+f)} \int_1^{+\infty}{e^{\lambda \sqrt{T}\sqrt{y}-\alpha_T y}\, \mathrm{d}y }.
\end{equation}
With the change of variable given by $u=\sqrt{y}-\frac{\lambda \sqrt{T}}{2\alpha_T}$, it becomes
$$\begin{array}{lcl}
K_T &\geq& \frac{2^{1-f} \left(\beta_T\right)^f}{\Gamma(1+f)} e^{\frac{\lambda^2 T}{4\alpha_T}}\int_{1-\frac{\lambda \sqrt{T}}{2\alpha_T}}^{+\infty}{e^{-\alpha_T u^2}(u+\frac{\lambda \sqrt{T}}{2\alpha_T})\, \mathrm{d}u }\\
&\geq& \frac{2^{1-f} \left(\beta_T\right)^f}{\Gamma(1+f)} e^{\frac{\lambda^2 T}{4\alpha_T}}\left(\int_{1-\frac{\lambda \sqrt{T}}{2\alpha_T}}^{+\infty}{e^{-\alpha_T u^2} u\, \mathrm{d}u } +\frac{\lambda \sqrt{T}}{2\alpha_T}\int_{1-\frac{\lambda \sqrt{T}}{2\alpha_T}}^{+\infty}{e^{-\alpha_T u^2}\, \mathrm{d}u }\right)\\
\end{array}$$
 By dominated convergence, the first integral in the last expression tends to zero as $T$ goes to infinity and the second integral converges to the positive constant $2\sqrt{\frac{\pi}{d-b}}$. It gives, for $T$ large enough, \begin{equation}K_T \geq \sqrt{\frac{\pi}{d-b}} \frac{2^{1-f} \left(\beta_T\right)^f}{\Gamma(1+f)} \frac{\lambda \sqrt{T}}{2\alpha_T} e^{\frac{\lambda^2 T}{4\alpha_T}}.
 \end{equation}
 Moreover, we establish the upper bound with formula 3.462(1) of \cite{GR} as in the previous proof.
$$\begin{array}{lcl}
K_T &\leq& \frac{\left(\beta_T\right)^f}{2^f \Gamma(1+f)}  \left(2\left(2\alpha_T\right)^{-(\gamma+g)} \Gamma\left(2\gamma+2g\right) e^{\frac{(\lambda\sqrt{T}+\beta_T)^2}{8\alpha_T}} D_{-2(\gamma+g)} \left(\frac{-\beta_T-\lambda\sqrt{T}}{\sqrt{2\alpha_T}}\right)\right.\\
& & \left. \, \, \, \, \, \, \, \, \, \, \, \, \, \, \, \, \, \, \, \, \, \, \, \, \, \, \, \, \, \, \displaystyle - \int_0^1{e^{\beta_T \sqrt{y} + \lambda \sqrt{Ty}-\alpha_T y}\,  y^{\gamma+g-1} \, \mathrm{d}y}\right)\\
& \leq & 2^{1-f} \left(\beta_T\right)^f \left(2\alpha_T\right)^{-(\gamma+g)} \frac{\Gamma\left(2\gamma+2g\right)}{\Gamma(1+f)}  \,  e^{\frac{(\lambda\sqrt{T}+\beta_T)^2}{8\alpha_T}} D_{-2(\gamma+g)} \left(\frac{-\beta_T-\lambda\sqrt{T}}{\sqrt{2\alpha_T}}\right).
\end{array}$$
But $-\beta_T-\lambda\sqrt{T} <0$ so, by formula 9.246(2) in \cite{GR} $$D_{-2(\gamma+g)} \left(\frac{-\beta_T-\lambda\sqrt{T}}{\sqrt{2\alpha_T}}\right) \sim \frac{\sqrt{2\pi}}{\Gamma\left(2\gamma+2g\right)} \left(\frac{\beta_T+\lambda\sqrt{T}}{\sqrt{2\alpha_T}} \right)^{2\gamma+2g-1} e^{\frac{(\lambda\sqrt{T}+\beta_T)^2}{8\alpha_T}}.$$ This gives, for $T$ large enough, 
\begin{equation}
K_T \leq \frac{2^{2-f} \sqrt{2\pi} \left(\beta_T\right)^f}{\Gamma(1+f)} \left(\beta_T+\lambda\sqrt{T} \right)^{2\gamma+2g-1}e^{\frac{(\lambda\sqrt{T}+\beta_T)^2}{4\alpha_T}}.
\end{equation}
This easily leads to the announced result. \demend
 
 \section*{Appendix D: proof of Lemma ~\ref{LDPtriplet}}
\renewcommand{\thesection}{\Alph{section}}
\renewcommand{\theequation}{\thesection.\arabic{equation}}
\setcounter{section}{4}
\setcounter{equation}{0}

 We call $L$ the pointwise limit of the cumulant generating function of the triplet $\left(X_T/T,S_T,\Sigma_T\right)$. We notice that $L(\lambda,\mu,\nu)=\Lambda(\lambda,\mu,\nu,0)$ where $\Lambda$ is given by Proposition~\ref{CGFquadruplet}. We easily deduce from Lemma~\ref{steep} that the function $L$ is steep. We apply the G\"artner-Ellis theorem: the rate function is given by the Fenchel-Legendre transform of $L$ on its effective domain.
 \begin{equation}I(x,y,z) = \underset{\lambda \in \mathbb{R}, \, \mu<\frac{b^2}{8}, \, \nu < \frac{(a-2)^2}{8}}{\sup} \left\lbrace x\lambda + y \mu + z\nu- L(\lambda,\mu,\nu) \right\rbrace.
 \end{equation}
For $y$ and $z$ we recognize the exact same term than in Appendix B. So with the same argument, we show that \begin{equation}\text{ for } y \leq 0, z\leq 0 \text{ or } yz-1 \leq 0, \: \: I(x,y,z)= +\infty.
\end{equation}
 Besides, if $x <0$ then for $\lambda$ tending to $-\infty$, $\lambda x \to +\infty$ and $I(x,y,z)=+\infty$ because $\Lambda$ does not depend on $\lambda$ for $\lambda<0$. Moreover, for $x>0$, the term on $\lambda$ is always negative for $\lambda$ negative and sometimes positive if $\lambda$ is positive. So the supremum is necessarily reached for $\lambda >0$. We finally have to calculate:
\begin{equation}I(x,y,z) = \underset{\lambda>0, \, \mu<\frac{b^2}{8}, \, \nu < \frac{(a-2)^2}{8}}{\sup} \left\lbrace x\lambda + y \mu + z\nu- \Lambda(\lambda,\mu,\nu) \right\rbrace
\end{equation}
with $x\geq 0$, $y>0$, $z>0$ and $yz-1>0$.
We do exactly as in Appendix A, we are looking for critical points and the calculations are very similar. We find: $\lambda_0= \frac{x}{2}(\frac{z(x^2+2)}{yz-1}-b)$, $\mu_0=\frac{1}{8}\left(b^2-\frac{z^2(x^2+2)^2}{(yz-1)^2}\right)$ and $\nu_0=\frac{1}{8}\left((a-2)^2-\frac{(x^2+2)^2}{(yz-1)^2} \right)$, which leads easily to $I$.
\demend

\section*{Appendix E: Proof of Lemma ~\ref{LDPtriplet2}}
\renewcommand{\thesection}{\Alph{section}}
\renewcommand{\theequation}{\thesection.\arabic{equation}}
\setcounter{section}{5}
\setcounter{equation}{0}

The pointwise limit of the cumulant generating function of the considered triplet is easily given by $\Lambda(0,\mu,\nu,\gamma)$ where $\Lambda$ is defined in Lemma ~\ref{CGFquadruplet}:
\begin{equation}\Lambda(0,\mu,\nu,\gamma)=\left\lbrace \begin{array}{ll}
-\frac{d}{2}\left(1+f\right)-\frac{ab}{4}+\frac{\gamma^2}{2f+a+2} & \text{ if } \gamma <0\\
-\frac{d}{2}\left(1+f\right)-\frac{ab}{4} & \text{ else. }
\end{array}   \right.
\end{equation}
We recall that $\displaystyle d=\sqrt{b^2-8\mu}$ and $\displaystyle f=\frac{1}{2}\sqrt{(a-2)^2-8\nu}$.
Using the G\"artner-Ellis theorem, we have \begin{equation}\wt{I}(y,z,t)=\underset{\gamma \in \mathbb{R}, \, \mu<\frac{b^2}{8}, \, \nu < \frac{(a-2)^2}{8}}{\sup} \left\lbrace y\mu+z\nu+t\gamma - \Lambda(0,\mu,\nu,\gamma) \right\rbrace.
\end{equation}
With the same argument than for the other couple of simplified estimators, we show that $\wt{I}(y,z,t)=+\infty$ for $y<0$, $z<0$ or $yz-1<0$. We also notice that for $t>0$  the expression inside the supremum tends to infinity as $\gamma$ goes to infinity. So $\wt{I}(y,z,t)=+\infty$ for $t>0$. Besides, for $t \leq 0$, the part involving $\gamma$ is always negative for $\gamma \geq 0$ and sometimes positive for $\gamma<0$. It implies that the supremum is necesseraly reached for $\gamma<0$. Replacing $\mu$ and $\nu$ by their expression on $d$ and $f$, we obtain, for $y>0$, $z>0$, $yz-1>0$ and $t\leq 0$,
$$\wt{I}(y,z,t)= \underset{\gamma<0, \, d>0, \, f>0}{\sup} \left\lbrace y \frac{b^2-d^2}{8}+z \frac{(a-2)^2-4f^2}{8}+t\gamma +\frac{ab}{4}+\frac{d}{2}\left(1+f\right)-\frac{\gamma^2}{2f+a+2} \right\rbrace.$$
We investigate critical points. We obtain 
\begin{equation} f_0=\frac{yt^2+2}{2(yz-1)}, \, \, d_0=\frac{t^2+2z}{yz-1} \, \text{ and } \, \gamma_0=\frac{yt(t^ 2+2z)}{2(yz-1)}+\frac{at}{2}.
\end{equation}
Replacing it into the expression of $\wt{I}$, we easily get the announced result.\demend


\nocite{*}
\bibliographystyle{acm}
\bibliography{biblio}

\end{document}